\documentclass[11pt]{amsart}

\usepackage{latexsym,rawfonts}
\usepackage{amsfonts,amssymb}
\usepackage{amsmath,amsthm,enumerate}

\usepackage{microtype}

\usepackage[plainpages=false,hypertexnames=false]{hyperref}
\usepackage{graphicx}
\usepackage{longtable}

\newcommand{\R}{\mathbb{R}}

\newcommand{\cC}{\mathcal{C}}

\newcommand{\cN}{\mathcal{N}}

\newcommand{\psum}{{+_{\negthinspace\kern-2pt p}}\,}
\newcommand{\qsum}[1]{{+_{\negthinspace\kern-2pt #1}}\,}
\newcommand{\dpsum}{{\tilde+_{\negthinspace\kern-1pt p}}\,}
\newcommand{\dqsum}[1]{{\tilde+_{\negthinspace\kern-1pt #1}}\,}
\newcommand{\lsub}[1]{\hskip -1.5pt\lower.5ex\hbox{$_{#1}$}}

\numberwithin{equation}{section}

\newtheorem{theo}{Theorem}[section]

\newtheorem{lem}[theo]{Lemma}

\newtheorem{rem}[theo]{Remark} \theoremstyle{definition}

\begin{document}

\title[Capillary ellipsoid and curvature problems]{Capillary John ellipsoid theorem with applications to capillary curvature problems}
\author[J. Hu]{Jinrong Hu}
\address{Institut f\"{u}r Diskrete Mathematik und Geometrie, Technische Universit\"{a}t Wien, Wiedner Hauptstrasse 8-10, 1040 Wien, Austria
 }
\email{jinrong.hu@tuwien.ac.at}

\author[Y. Hu]{Yingxiang Hu}
\address{School of Mathematical Sciences, Beihang University, Beijing 100191, China}
\email{huyingxiang@buaa.edu.cn}

\author[B. Yang]{Bo Yang}
\address{Department of Mathematical Sciences, Tsinghua University, Beijing 100084, P.R. China}
\email{ybo@tsinghua.edu.cn}
\address{Institut f\"{u}r Diskrete Mathematik und Geometrie, Technische Universit\"{a}t Wien, Wiedner Hauptstra{\ss}e 8-10, 1040 Wien, Austria}
\email{bo.yang@tuwien.ac.at}

\begin{abstract}
In this paper, we apply a capillary John ellipsoid theorem for capillary convex bodies in the Euclidean half-space $\overline{\mathbb{R}^{n+1}_{+}}$. This theorem yields a non-collapsing estimate for capillary hypersurfaces, which provides a new approach to obtaining $C^{0}$ estimates for solutions to some capillary curvature problems (including the capillary $L_{p}$ Christoffel-Minkowski problem and the capillary $L_{p}$ curvature problem), based on the corresponding gradient estimates. As an application, we study the capillary $L_{p}$ dual Minkowski problem. A gradient estimate, together with the non-collapsing estimate and a $C^2$ estimate, yields existence for every $1<p<q<\infty$ with positive smooth even data and an acute contact angle. The same estimates also yield the eigenvalue problem for $p=q>1$, with uniqueness up to dilation. We further obtain existence and uniqueness for $p>q$ without an upper restriction on $q$.
\end{abstract}

\makeatletter
\@namedef{subjclassname@2020}{\textup{2020} Mathematics Subject Classification}
\makeatother

\thanks{J. Hu was supported by the Austrian Science Fund (FWF): 10.55776/ESP1358925. Y. Hu was supported by the National Key Research and Development Program of China 2021YFA1001800 and NSFC grant 12671060. B. Yang was supported by China Postdoctoral Science Foundation No.~2024M751605.}
\maketitle
\baselineskip18pt \parskip3pt

\section{Introduction}
Capillary geometry has emerged as a key tool for modeling liquid interfaces in contact with solid boundaries, generating substantial research interest (cf. \cite{HWYZ24, KLS25, MWWX25,WWX24}). These investigations naturally lead to the study of capillary hypersurfaces that exhibit a prescribed contact angle $\theta \in (0, \frac{\pi}{2})$ along their supporting boundary. Such hypersurfaces provide a geometric framework for studying free boundary problems and surface-tension phenomena. At the intersection of convex geometry, capillary geometry, and boundary value problems for partial differential equations lies the capillary Minkowski problem. This problem asks whether a given function defined on a spherical cap can be realized as the Gauss curvature of a convex capillary hypersurface. Introduced and solved by Mei, Wang and Weng \cite{MWW25a}, it may be regarded as the capillary analogue of the classical Minkowski problem.

The introduction of the capillary Minkowski problem has generated substantial research interest and spurred rapid development in the field. In particular, several important generalizations have been proposed and systematically investigated. The $L_p$ capillary Minkowski problem, which seeks a convex capillary hypersurface with prescribed $L_p$ surface area measure, has been extensively studied for various ranges of the parameter $p$, as discussed in \cite{DGLL26,HHI25,HI25a,MWW25b}. As a further extension, the Orlicz capillary Minkowski problem has been investigated in \cite{LL26,WZ25}, providing a Robin-type boundary analogue of the classical Orlicz-Minkowski problem. More recently, the capillary Christoffel--Minkowski-type problem, concerning the prescribed area measures of capillary convex bodies, has also attracted attention (see, e.g., \cite{HI25,HIS25,MWW25,MWW25c}). In addition to these Minkowski-type problems, the prescribed capillary curvature problem, which addresses the existence of convex capillary hypersurfaces with prescribed Weingarten curvature, has also been explored (cf. \cite{HI26,MWW25}). Collectively, these developments underscore that capillary geometry has rapidly evolved into a vibrant research area. In particular, the results on curvature problems in Euclidean space  (see, e.g., \cite{M897, M903, A38, A39,  N53, CY76, P78,L93,GG02,GM03,HLYZ10} and references therein) have been extended to the Euclidean half-space.

The key to solving curvature problems lies in obtaining regularity estimates for the solutions, in which the classical John's theorem for convex bodies in $\mathbb{R}^{n+1}$ plays a crucial role (see, e.g., \cite[Chap. 10]{S14}). In this paper, we apply a capillary John ellipsoid theorem (see Theorem \ref{Thm-ce}). Using this theorem, we obtain a new non-collapsing estimate for capillary hypersurfaces as follows.
\begin{lem}\label{TREQ}
Let $\theta\in (0,\frac{\pi}{2})$. Let $\Sigma$ be an even, smooth, strictly convex $\theta$-capillary hypersurface with positive capillary support function $h$. If $h$ satisfies
\begin{equation}\label{GY}
      \frac{|\nabla h|^2}{h^{\gamma}} \leq N (\max_{\mathcal{C}_{\theta}} h)^{2-\gamma}, \quad \text{in \ $\mathcal{C}_{\theta}$}
\end{equation}
for some positive constants $\gamma$ and $N$. Then the following non-collapsing estimate holds
\begin{equation}\label{maxmin}
    \frac{\max_{\mathcal{C}_{\theta}}h}{\min_{\mathcal{C}_{\theta}}h}\leq C,
\end{equation}
where the constant $C$ depends only on $n,\theta,\gamma$ and $N$.
\end{lem}

This lemma can be seen as a capillary analogue of \cite[Lem. 3.1]{Gu23} in the even case. The non-collapsing estimate \eqref{maxmin} provides a novel framework for obtaining $C^0$ estimates for solutions to capillary curvature problems such as the capillary $L_p$ Christoffel-Minkowski problem and the capillary $L_p$ curvature problem (see, e.g., \cite{HI26,HI25,MWW25}), based on the gradient estimates. The $C^0$ estimate plays a crucial role in deriving higher-order regularity estimates for curvature equations (for further details, see Section \ref{sec4}).

In addition to the capillary curvature problems mentioned above, we consider another class of capillary curvature problems, the capillary $L_p$ dual Minkowski problem in $\overline{\mathbb{R}^{n+1}_{+}}$. This problem asks for a capillary convex hypersurface whose capillary support function $h:\mathcal{C}_\theta \rightarrow (0,\infty)$ satisfies the following fully nonlinear partial differential equation with a Robin boundary condition:
\begin{equation}\label{Mong-Eq}
\left\{
\begin{array}{l@{\ }l}
\det(\nabla^{2}h + hI) = f h^{p-1} (h^{2} + |\nabla h|^{2})^{\frac{n+1-q}{2}}, & \quad  {\rm in} \ \mathcal{C}_{\theta}, \\
\nabla_{\mu}h=\cot \theta \,h, & \quad  {\rm on} \ \partial\mathcal{C}_{\theta},
\end{array}\right.
\end{equation}
where $\mu$ is the unit outward conormal of $\partial\mathcal{C}_{\theta}$ in $\mathcal{C}_{\theta}$.

When $q=n+1$, \eqref{Mong-Eq} corresponds to the capillary $L_p$ Minkowski problem; when $\theta=\pi/2$, it becomes a hemispherical Neumann problem. Its relation to the closed $L_p$ dual Minkowski problem (see, e.g., \cite{HLYZ16, LYZ18, HZ18}) uses reflection across the equator and reflection-compatible data; the endpoint $\theta=\pi/2$ is not included in the estimates proved here. Regarding the solvability of \eqref{Mong-Eq}, Gao \cite{G26a} established the existence and uniqueness of solutions under the conditions $p > q$ and $q\leq1$, where $q\leq 1$ ensures the completion of the interior $C^2$ estimate for \eqref{Mong-Eq}. We remove this restriction on $q$ and also treat $1<p\leq q$.

We consider the following new test function
\[
  Q=\log \sigma_1+Ah+B|\nabla u|^{2}\]
in place of the function
\[
   P_1=\sigma_1+Mh+\frac{1}{2}|\nabla h|^{2}
\]
used in \cite{G26a}, where $\sigma_{1}$ is the sum of the principal radii of curvature of a capillary convex hypersurface and $u=h/\ell$ (where $\ell$ is the capillary support function of the unit spherical cap; see Section \ref{sec2} for its definition). A refined $C^2$ estimate then yields the existence of solutions to the capillary $L_p$ dual Minkowski problem for $p>q$ in $\overline{\mathbb{R}^{n+1}_{+}}$ without requiring $q\leq 1$.

It is worth mentioning that the auxiliary function $Q$ can be employed to yield the required $C^2$ estimate, both in the interior and on the boundary, without any restriction on $q$.  For the $C^0$ estimate of solutions to \eqref{Mong-Eq} with $p>q$, upper and lower bounds can be directly obtained via the maximum principle; see \cite{G26a} for details. However, when $p\leq q$, the $C^0$ estimate is more challenging, as it may require a capillary gradient estimate in higher dimensions. Due to the presence of the gradient term $\nabla h$ on the right-hand side of \eqref{Mong-Eq} and the boundary condition, deriving the gradient estimate becomes significantly more delicate than in \cite{CH25,HI24, HI26, MWW24,MWW25}. Under \eqref{eq:gamma-condition}, Lemma \ref{gra-esi} provides a gradient estimate for arbitrary $q$ by estimating the term containing $(n+1-q)$ from both sides. Together with Lemma \ref{TREQ} and the $C^2$ estimate, a two-stage degree argument yields existence for $1<p<q<\infty$. Normalized approximation then settles the critical case $p=q>1$.

The main results are as follows.
\begin{theo}\label{maintheo2}
Let $n\geq2$ and $\theta \in (0,\frac{\pi}{2}) $. Let $f$ be a positive function on $\mathcal{C}_{\theta}$ that is smooth up to the boundary.
\begin{enumerate}[(i)]
\item If $1<p<q<\infty$ and $f$ is even, then there exists an even, positive, smooth and strictly convex solution $h$ to Eq. \eqref{Mong-Eq}.
\item For every fixed $p,q\in\mathbb R$ with $p>q$, there exists a unique, positive, smooth and strictly convex solution $h$  to Eq. \eqref{Mong-Eq}.
\item If $p=q>1$ and $f$ is even, then there exist a positive, even, smooth, strictly convex solution $h$, unique up to positive dilation, and a unique positive constant $C^*$ such that
\begin{equation}\label{p=q}
\left\{
\begin{array}{l@{\ }l}
\det(\nabla^{2}h + hI) = C^*f h^{p-1} (h^{2} + |\nabla h|^{2})^{\frac{n+1-p}{2}}, & \quad  {\rm in} \ \mathcal{C}_{\theta}, \\
\nabla_{\mu}h=\cot \theta \,h, & \quad  {\rm on} \ \partial\mathcal{C}_{\theta}.
\end{array}\right.
\end{equation}
\end{enumerate}
\end{theo}

\begin{rem}
The estimates in (i) are uniform on compact subsets of $1<p<q<\infty$ and $\theta\in (0,\frac{\pi}{2})$. They need not remain uniform as $p\downarrow1$, $q-p\downarrow0$ or $q\to\infty$; part (iii) uses normalization. For $p>q$, Gao \cite[Thm.~1.1]{G26b} also recently obtained existence and uniqueness in dimensions $n\geq3$, without an upper restriction on $q$, using the test function $P_2=\log\sigma_1+Ah+B|\nabla h|^2+\log\sigma_{n-1}$.
\end{rem}
The structure of this paper is as follows. In Section \ref{sec2}, we recall some basic facts about capillary hypersurfaces. In Section \ref{sec3}, we state a capillary John ellipsoid theorem, which serves as the foundation for a new non-collapsing estimate developed in Section \ref{sec4}. In Section \ref{sec5}, the proof of Theorem \ref{maintheo2} is presented.

\section{Basics of capillary geometry}\label{sec2}
Let $\{E_i\}_{i=1}^{n+1}$ denote the standard orthonormal basis of $\mathbb{R}^{n+1}$, and define the Euclidean half-space as
\begin{equation*}
\mathbb{R}^{n+1}_{+} = \{ y \in \mathbb{R}^{n+1} : \langle y, E_{n+1} \rangle > 0 \}.
\end{equation*}
Let $\Sigma \subset \overline{\mathbb{R}^{n+1}_{+}}$ be a smooth, properly embedded, compact hypersurface with boundary $\partial \Sigma \subset \partial \mathbb{R}^{n+1}_{+}$. A hypersurface $\Sigma$ is called a capillary hypersurface with constant contact angle $\theta \in (0, \pi)$ if
\begin{equation}
\langle \nu, e \rangle= \cos(\pi - \theta) \ \text{along $\partial \Sigma$},\notag
\end{equation}
where $e := -E_{n+1}$ and $\nu$ denotes the outward unit normal of $\Sigma$ in $\mathbb{R}^{n+1}$. If we additionally let $\mu$ be the unit outward conormal of $\partial\Sigma$ in $\Sigma$ and $\bar{\nu}$ be the unit normal of $\partial\Sigma$ in $\partial\mathbb{R}^{n+1}_{+}$ such that \{$\nu,\mu$\} and \{$\bar{\nu},e$\} have the same orientation in the normal bundle of $\partial\Sigma\subset\overline{\mathbb{R}^{n+1}_{+}}$, then the following relation holds:
\begin{equation}\label{normal transform}
\left\{\begin{aligned}
    \nu={}&-\cos\theta\, e+\sin\theta\, \bar{\nu},\\
	\mu={}&\sin\theta \,e+\cos\theta \,\bar{\nu}.
	\end{aligned}\right.
\end{equation}

A compact convex set $\widehat{\Sigma}\subset\overline{\mathbb{R}^{n+1}_{+}}$ with nonempty interior is called a capillary convex body if its boundary consists of a strictly convex capillary hypersurface $\Sigma$ and a portion of the supporting hyperplane $\partial \mathbb{R}^{n+1}_{+}$.

For any $r > 0$, the capillary spherical cap of radius $r$ with constant contact angle $\theta \in (0,\pi)$ is defined as
\begin{equation} \label{s2:def-C_theta}
\mathcal{C}_{\theta,r} = \{\zeta \in \overline{\mathbb{R}^{n+1}_{+}} : |\zeta-r \cos \theta \, e| = r \}.\notag
\end{equation}
We write $\mathcal{C}_{\theta}=\mathcal{C}_{\theta,1}$.

The Gauss image $\nu(\Sigma)$ of a strictly convex capillary hypersurface $\Sigma$ lies in  $S_{\theta}:=\{u\in\mathbb{S}^{n}|u_{n+1}\geq\cos\theta\}$, and the capillary Gauss map $\tilde{\nu}: \Sigma \to \mathcal{C}_\theta$ defined by
\begin{equation}
\tilde{\nu} := \nu + \cos\theta\, e\notag
\end{equation}
is a diffeomorphism. Its inverse $\tilde{\nu}^{-1}: \mathcal{C}_\theta \to \Sigma$ is the inverse capillary Gauss map.

Taking $X=\tilde{\nu}^{-1}$, the capillary support function of a strictly convex capillary hypersurface $\Sigma$, denoted by $h = h_{\Sigma} : \mathcal{C}_\theta \to \mathbb{R}$, is defined as
\begin{equation*}
h(\xi) := \langle X(\xi) , \nu\big(X(\xi)\big)\rangle = \langle\tilde{\nu}^{-1}(\xi)  ,\xi - \cos \theta\, e\rangle, \quad \forall \, \xi \in \mathcal{C}_\theta. 
\end{equation*}
Then $h$ satisfies the capillary boundary condition (cf. \cite[Lem. 2.4]{MWWX25})
\begin{equation}
\nabla_{\mu} h = \cot \theta\, h, \quad \text{on } \partial \mathcal{C}_\theta. \notag
\end{equation}
Denote by $\ell(\xi)$ the capillary support function of $\mathcal{C}_\theta$. When $\Sigma=\mathcal{C}_{\theta}$, we have $X(\xi)=\xi$ for all $\xi\in \mathcal{C}_\theta$; hence,
\begin{equation}
 \ell(\xi)=\sin^2\theta+\cos \theta \langle \xi, e \rangle.  \notag
\end{equation}
Moreover, in the acute-angle range $\theta \in (0,\pi/2)$, $\ell$ satisfies the following relations
\begin{equation}\label{eq:fixed-cap-identities}
1-\cos\theta\leq\ell\leq\sin^2\theta,\qquad
\nabla^2\ell+\ell I=I,\qquad \ell_\mu=\cot\theta\,\ell.
\end{equation}
For our later use, set $u=h/\ell$. The Robin condition implies
\begin{equation}\label{eq:Neumann-identities}
u_\mu=0,\qquad
u_{\alpha\mu}=-\cot\theta\,u_\alpha, \qquad 1\leq\alpha\leq n-1.
\end{equation}

Let $\hat{h}:=\hat{h}_{\widehat{\Sigma}}: \mathbb{S}^n\rightarrow \mathbb{R}$ be the (standard) support function of $\widehat{\Sigma}$. Then
\begin{equation}\label{hxi}
h(\xi)=\hat{h}(\xi-\cos \theta \,e),\quad\forall\xi\in\mathcal{C}_{\theta}.
\end{equation}

In a local orthonormal frame $\{e_{i}\}^n_{i=1}$ on $\mathcal{C}_{\theta}$  with respect to its standard induced metric, the Gauss curvature $\mathcal{K}$ of $\Sigma$ is expressed by
\begin{equation}
 \mathcal{K}\left( X(\xi)\right) = \frac{1}{\det\left( \nabla^{2} h(\xi) + h(\xi) I \right)} \notag.
\end{equation}

Let $\xi=(\xi', \xi_{n+1})\in \mathcal{C}_\theta$ where $\xi'=(\xi_1,\ldots, \xi_{n})$. A smooth function  $f: \mathcal{C}_\theta \rightarrow \mathbb{R}$ is called \emph{rotationally symmetric} if it satisfies
\begin{equation}
f(A\xi', \xi_{n+1}) = f(\xi', \xi_{n+1}), \ \forall\, \xi \in \mathcal{C}_\theta, \,  A\in O(n), \notag
\end{equation}
where $O(n)$ denotes the orthogonal group of $\mathbb R^n$. A smooth function  $f: \mathcal{C}_\theta \to \mathbb{R}$ is called \emph{even} if
\begin{equation}
f(-\xi', \xi_{n+1}) = f(\xi',\xi_{n+1}), \ \forall\, \xi \in \mathcal{C}_\theta.  \notag
\end{equation}
A strictly convex capillary hypersurface is called rotationally symmetric or even,  if its capillary support function is rotationally symmetric or even, respectively.

\begin{lem}\label{lem:realization-strict}
Let $\theta\in(0,\frac{\pi}{2})$ and let $h$ be a positive smooth function on $\mathcal{C}_\theta$ satisfying
\[
b:=\nabla^2h+hI>0,\qquad h_\mu=\cot\theta\,h.
\]
Then $h$ is the capillary support function of a smooth strictly convex capillary body $\widehat\Sigma$. If $h$ is even, then $\widehat\Sigma$ is even. Moreover,
\begin{equation}\label{eq:geometric-gradient-bound}
|\nabla h|\leq \sqrt{1+\csc^2\theta}\,\max h.
\end{equation}
\end{lem}
\begin{proof}
The first assertion is precisely the correspondence between capillary convex functions and capillary convex bodies in \cite[Prop. 2.6]{MWWX25}. If $Q(x',x_{n+1})=(-x',x_{n+1})$, the capillary support function of $Q\widehat\Sigma$ is $h(-\xi',\xi_{n+1})$. Thus, when $h$ is even, the uniqueness in that correspondence implies $Q\widehat\Sigma=\widehat\Sigma$.

Let $M=\max_{\mathcal{C}_\theta} h$ and let $\hat h$ be the standard support function of $\widehat\Sigma$. For $x=(x',z)\in\widehat\Sigma$,  in the directions $E_{n+1}$ and $(\sin\theta\,x'/|x'|,\cos\theta)$, we find
\[
0\leq z\leq M,\qquad |x'|\leq M/\sin\theta,
\]
with the second inequality being immediate when $x'=0$. By \cite[Lem. 2.4(1)]{MWWX25}, the inverse capillary Gauss map satisfies $X=\nabla h+h\,(\xi-\cos\theta \,e)$, and hence $|X|^2=h^2+|\nabla h|^2$. The preceding cylinder bound proves \eqref{eq:geometric-gradient-bound}.
\end{proof}

For later regularity arguments, we record the following gnomonic coordinates. Translate $\mathcal{C}_\theta$ to $S_\theta$, write $h(v)$ for $h(v+\cos\theta\, e)$, and set
\[
a(y)=(1+|y|^2)^{-1/2},\qquad v(y)=a(y)(y,1),\qquad
H(y)=h(v(y))/a(y),\qquad |y|\leq\tan\theta.
\]
Differentiation of the degree-one homogeneous extension gives
\begin{equation}\label{eq:gnomonic-Hessian}
D^2H(y)[z,z]=a(y)b_{v(y)}\bigl(P_{v(y)}(z,0),P_{v(y)}(z,0)\bigr)>0, \quad z\neq 0,
\end{equation}
where $P_v$ denotes orthogonal projection onto $T_v\mathbb S^n$. Since the Gram matrix of this projection on horizontal vectors is $I-v'\otimes v'$, whose determinant is $a^2$, we also have
\begin{equation}\label{eq:gnomonic-determinant}
\det D^2H=a^{n+2}\det b.
\end{equation}
\section{A capillary John ellipsoid theorem}
\label{sec3}
We write $x=(x',x_{n+1})\in\mathbb{R}^{n}\times\mathbb{R}$. For $a,b>0$, define the rotationally symmetric ellipsoid in $\mathbb{R}^{n+1}$ by
\begin{equation}
E=E(a,b)=\{a^2|x'|^2+b^2 x_{n+1}^2\leq  1\}
\end{equation}
and set $\eta=a/b$. We translate $E$ downwards by $-\tau^{\ast} E_{n+1}$ for $\tau^{\ast}>0$ so that the hypersurface $(\partial E-\tau^{\ast} E_{n+1})\cap\{x_{n+1}\ge 0\}$ meets $\{x_{n+1}=0\}$ at the constant angle $\theta\in(0,\frac{\pi}{2})$:
\begin{equation}
\begin{split}
\label{den-tau}
\tau^{\ast}&=\tau^{\ast}(a,b):=\frac{a\cos \theta}{b\sqrt{a^2\cos^2 \theta+b^2\sin^2 \theta}},\\
\lambda&=\lambda_\eta:=b\,\tau^{\ast}(a,b)=\frac{\eta\cos \theta}{\sqrt{\eta^2\cos^2 \theta+\sin^2 \theta}}   \notag.
\end{split}
\end{equation}
The translated $\theta$-capillary cap is denoted by
\begin{equation*}\label{Lab}
L=L(a,b):=(\partial E-\tau^{\ast} E_{n+1})\cap\{x_{n+1}\geq 0\}.
\end{equation*}
We write
\[
\widehat L(a,b):=(E-\tau^*E_{n+1})\cap\{x_{n+1}\geq0\}
\]
for the solid capillary ellipsoid enclosed by $L(a,b)$ and its planar base. All set inclusions below refer to this solid. We also define $R=R(a,b):=\max_{x\in L}|x'|$ and $H=H(a,b):=\max_{x\in L}{x_{n+1}}$, i.e., $R$ is the radius of the sphere $L\cap\{x_{n+1}=0\}$ and $H$ is the height of $L$. After a direct calculation, we see that
\begin{align}
R&=R(a,b)=\frac{1}{a}\frac{\sin\theta}{\sqrt{\sin^2\theta+\eta^2\cos^2\theta}},\label{eq-R}\\
H&=H(a,b)=\frac{1}{b}\left(1-\frac{\eta\cos\theta}{\sqrt{\sin^2\theta+\eta^2\cos^2\theta}}\right)\label{eq-H}.
\end{align}
For later use we only need the even case, so the required comparison can be proved directly at the level of convex bodies.

\begin{theo}\cite{HI26a}\label{Thm-ce}
Let $\widehat\Sigma\subset\overline{\mathbb{R}^{n+1}_{+}}$ be an even capillary convex body whose free boundary $\Sigma$ is smooth and strictly convex, with contact angle $\theta\in(0,\frac{\pi}{2})$. Set
\[
R^{\mathrm{out}}_{\Sigma}=\max_{x\in\Sigma}|x'|,
\qquad
R^{\mathrm{in}}_{\Sigma}=\min_{x\in\partial\Sigma}|x'|,
\qquad
H_{\Sigma}=\max_{x\in\Sigma}x_{n+1}.
\]
Then there exists $(a,b)\in(0,\infty)^2$ such that
\begin{equation*}
\widehat L(a,b)\subset\widehat{\Sigma}\subset
\left(\frac{3}{2}\frac{R^{\mathrm{out}}_{\Sigma}}
{R^{\mathrm{in}}_{\Sigma}}+3\right)\widehat L(a,b).
\end{equation*}
\end{theo}

\begin{proof}
Let $D=\widehat\Sigma\cap\{x_{n+1}=0\}$. Evenness and strict convexity imply
\[
B_{R^{\mathrm{in}}_{\Sigma}}^n(0)\subset D
\subset B_{R^{\mathrm{out}}_{\Sigma}}^n(0),
\]
and the unique highest point of $\Sigma$ is $(0,H_{\Sigma})$. At a point $x_0=(R^{\mathrm{in}}_{\Sigma}\omega,0)\in\partial\Sigma$ where the inner radius is attained, the horizontal outward normal of $D$ is $\omega$. Hence the capillary tangent plane at $x_0$ has normal $(\sin\theta\,\omega+\cos\theta\,E_{n+1})$ and intersects the vertical axis at height $R^{\mathrm{in}}_{\Sigma}\tan\theta$. Since this plane supports the convex body, strict convexity yields
\begin{equation}\label{eq:height-inner-radius}
H_{\Sigma}<R^{\mathrm{in}}_{\Sigma}\tan\theta.
\end{equation}

Put $R_0=\frac23R^{\mathrm{in}}_{\Sigma}$ and $H_0=\frac13H_{\Sigma}$. By \eqref{eq:height-inner-radius}, $0<H_0/R_0<\frac12\tan\theta$. Define
\[
\lambda=\frac{H_0/R_0}{\tan\theta-H_0/R_0},\qquad
a=\frac{\sqrt{1-\lambda^2}}{R_0},\qquad
b=\frac{1-\lambda}{H_0}.
\]
Then $0<\lambda<1$, and substitution in \eqref{den-tau}--\eqref{eq-H} yields $R(a,b)=R_0$ and $H(a,b)=H_0$.

Let $\mathcal C(R,H)$ denote the cone with base $B_R^n(0)$ and vertex $(0,H)$, and let $\mathcal Z(R,H)=B_R^n(0)\times[0,H]$. Convexity implies
\[
\mathcal C(R_0,H_0)\subset\widehat L(a,b)
\subset\mathcal Z(R_0,H_0),
\qquad
\mathcal C(R^{\mathrm{in}}_{\Sigma},H_{\Sigma})
\subset\widehat\Sigma
\subset\mathcal Z(R^{\mathrm{out}}_{\Sigma},H_{\Sigma}).
\]
The elementary cross-section comparison
\[
\mathcal Z(R_2,H_2)\subset
\left(\frac{R_2}{R_1}+\frac{H_2}{H_1}\right)\mathcal C(R_1,H_1)
\]
now gives
\[
\widehat L(a,b)\subset\mathcal Z(R_0,H_0)
\subset\mathcal C(R^{\mathrm{in}}_{\Sigma},H_{\Sigma})
\subset\widehat\Sigma
\]
and
\[
\widehat\Sigma\subset\mathcal Z(R^{\mathrm{out}}_{\Sigma},H_{\Sigma})
\subset\left(\frac{3}{2}\frac{R^{\mathrm{out}}_{\Sigma}}
{R^{\mathrm{in}}_{\Sigma}}+3\right)\mathcal C(R_0,H_0)
\subset\left(\frac{3}{2}\frac{R^{\mathrm{out}}_{\Sigma}}
{R^{\mathrm{in}}_{\Sigma}}+3\right)\widehat L(a,b).
\]
\end{proof}

\section{A non-collapsing estimate for capillary hypersurfaces}
\label{sec4}
In this section, we first apply Theorem \ref{Thm-ce} to prove Lemma \ref{TREQ}.

\begin{proof}[Proof of Lemma \ref{TREQ}.]  
We first consider the case $\gamma\geq2$. Write $M_h=\max_{\mathcal{C}_\theta} h$. By the hypothesis \eqref{GY},
\[
|\nabla\log h|^2
\leq N\left(\frac{h}{M_h}\right)^{\gamma-2}\leq N.
\]
Any two points of $\mathcal{C}_\theta$ can be joined through its pole by a piecewise geodesic of length at most $2\theta$. Integration therefore gives
\[
\log\frac{\max_{\mathcal{C}_\theta} h}{\min_{\mathcal{C}_\theta} h}\leq2\theta\sqrt N,
\qquad \frac{\max_{\mathcal{C}_\theta} h}{\min_{\mathcal{C}_\theta} h}\leq e^{2\theta\sqrt N}.
\]
It remains to prove the case $0<\gamma<2$.

We first obtain
\begin{equation}\label{eq-ub1}
R_\Sigma^{\rm out}/R_\Sigma^{\rm in}\leq C'(n,\theta,\gamma,N).
\end{equation}
Suppose \eqref{eq-ub1} fails for fixed $n,\theta,\gamma,N$. By dilation, choose $h_j$ with $\max_{\mathcal{C}_\theta} h_j=1$ and $R_j^{\rm out}/R_j^{\rm in}\to\infty$. Lemma \ref{lem:realization-strict} places their bodies $K_j$ in the fixed cylinder $\{|x'|\leq\csc\theta,\ 0\leq x_{n+1}\leq1\}$, so $R_j^{\rm in}\to0$. By compactness \cite{S14}, a subsequence converges to $K$ in Hausdorff distance. Its support function $h_K$ is the uniform limit of $h_j$ on $S_\theta$, with $\max_{S_\theta} h_K=1$.

At a closest base point $R_j^{\rm in}\omega_j$, the base normal is $\omega_j$. The contact angle condition determines the normal $v_j=(\sin\theta\,\omega_j,\cos\theta)$, with support value $h_j(v_j)=\sin\theta\,R_j^{\rm in}\to0$. After a further subsequence, $\omega_j\to\omega$ and $h_K(v_0)=0$, where $v_0=(\sin\theta\,\omega,\cos\theta)$. Set $a_\gamma=1-\gamma/2\in(0,1)$. Then \eqref{GY} yields $|\nabla(h_j^{a_\gamma})|\leq a_\gamma\sqrt N$. Integrating on $\mathcal{C}_\theta$ and passing to the uniform limit, we find
\begin{equation}\label{eq:boundary-zero-growth}
h_K(v)^{a_\gamma}\leq a_\gamma\sqrt N\,d_{S_\theta}(v,v_0).
\end{equation}

For every $(x',z)\in K$, evenness and $K\subset\{z\geq0\}$ imply
\[
\sin\theta\,\omega\cdot x'+\cos\theta\,z\leq0,\qquad
-\sin\theta\,\omega\cdot x'+\cos\theta\,z\leq0.
\]
Thus $z=0$ and $\omega\cdot x'=0$. If $n=1$, these identities already force $K=\{0\}$, contradicting $\max_{S_\theta} h_K=1$. For $n\geq2$, choose $x=(x',0)\in K$ with $x'\neq0$, and put $e'=x'/|x'|$. The unit-speed boundary curve
\[
v(t)=\bigl(\sin\theta[\cos(t/\sin\theta)\omega
                    +\sin(t/\sin\theta)e'],\cos\theta\bigr)
\]
satisfies $d_{S_\theta}(v(t),v_0)\leq t$ and $h_K(v(t))\geq x\cdot v(t)=|x'|t+O(t^3)$. This contradicts \eqref{eq:boundary-zero-growth}, whose upper bound is $O(t^{1/a_\gamma})=o(t)$. Hence \eqref{eq-ub1} holds.

Using Theorem \ref{Thm-ce} together with \eqref{eq-ub1}, we obtain a translated $\theta$-capillary cap $L(a,b)$ such that
\begin{equation}\label{inclusion}
\widehat L(a,b)\subset\widehat{\Sigma}\subset C''\widehat L(a,b),
\end{equation}
where $C''=\frac{3}{2}C'+3$.

Denote by $\varpi$ the capillary support function of $L(a,b)$. By \eqref{inclusion} and \eqref{hxi}, we have
\begin{equation}\label{hui}
   \varpi(\xi)\leq h(\xi)\leq C'' \varpi(\xi), \quad \forall \xi\in   \mathcal{C}_{\theta}.
\end{equation}
Using \eqref{hxi} again, we get
\begin{equation}
     \varpi(\xi)=\sqrt{\frac{|\xi'|^2}{a^2}+\frac{(\xi_{n+1}+\cos \theta\, )^2}{b^2}} -\tau^{\ast} (\xi_{n+1}+\cos \theta),   \quad \xi\in \mathcal{C}_{\theta}.     \notag
\end{equation}

Next we consider the following two cases.

{\em Case 1:} ($a>b$). In this case, $\eta>1$ and by a direct calculation, using \eqref{eq-R} and \eqref{eq-H}, we have
\begin{equation}\label{Eq-RH}
\frac{R}{H}=\frac{\sqrt{\sin^2\theta+\eta^2\cos^2\theta}+\eta\cos\theta}{\eta\sin\theta}\leq 2\cot\theta+\frac{1}{\eta}\leq 2\cot\theta+1.
\end{equation}
In addition,
\begin{equation*}\label{eq-inclu}
   \mathcal{C}(R,H)\subset \widehat L(a,b)\subset\widehat{\Sigma}\subset C''\widehat L(a,b)\subset\mathcal{C}(C''R,C''R\tan\theta),
\end{equation*}
Indeed, the supporting hyperplanes of $\widehat L(a,b)$ along its spherical base imply $|x'|\leq R-x_{n+1}\cot\theta$, which proves the last inclusion. The lower bound below is the minimum of the support function of the inner cone, namely its distance from the origin to a supporting hyperplane. Therefore, we have
\begin{equation}\label{eq-uplb}
\max_{\mathcal{C}_{\theta}} h\leq C''R\max\{\tan\theta,1\},\quad \min_{\mathcal{C}_{\theta}} h\geq\frac{RH}{\sqrt{R^2+H^2}},
\end{equation}
where $\frac{RH}{\sqrt{R^2+H^2}}$ is the distance from the origin  to the lateral surface of the cone $\mathcal{C}(R,H)$. Thus, by \eqref{Eq-RH} and \eqref{eq-uplb} we obtain
\begin{equation*}\label{pin-2}
    \frac{\max_{\mathcal{C}_{\theta}}h}{\min_{\mathcal{C}_{\theta}}h}\leq C''\max\{\tan\theta,1\}\sqrt{\left(\frac{R}{H}\right)^2+1}\leq C''\max\{\tan\theta,1\}  \sqrt{(2\cot\theta+1)^2+1}.
\end{equation*}

{\em Case 2:} ($a\leq b$). Set $|\xi'|=t$ with $0\leq t \leq \sin \theta$ and write $\varpi(\xi):=q(t)$, where
\begin{equation*}
\begin{split}
\label{ARE}
   q(t):&=\sqrt{\frac{t^2}{a^2}+\frac{1-t^2}{b^2}}-\tau^{\ast}\sqrt{1-t^2}\geq\frac{t}{a}-\tau^{\ast}\sqrt{1-t^2}.
   \end{split}
\end{equation*}
Then for any $\tilde{t}\in \big[0,\left( \frac{b}{a}\right)^{\frac{2-\gamma}{2}}\sin\theta \big]$, we have
\begin{equation}\label{KI}
    \tilde{ t}\left(\frac{1}{a}\right)^{\frac{\gamma}{2}}  \left(\frac{1}{b} \right)^{\frac{2-\gamma}{2}} \leq q\left(\tilde{t}\left(\frac{a}{b}\right)^{\frac{2-\gamma}{2}} \right) +\frac{1}{b},
\end{equation}
where we used $\tau^{\ast}\leq \frac{ 1}{b}$. Fix $\omega\in\mathbb S^{n-1}$ and set $\xi(t)=(t\omega,\sqrt{1-t^2}-\cos\theta)\in\mathcal{C}_\theta$. Denote $p(t):=h(\xi(t))^{\frac{2-\gamma}{2}}$. Using \eqref{GY}, we get
\begin{equation}
\begin{split}
   \Big|\frac{d}{dt}p(t)\Big|&\leq \frac{2-\gamma}{2}\frac{1}{\sqrt{1-t^2}}N^{\frac{1}{2}} (\max_{\mathcal{C}_{\theta}} h)^{1-\frac{\gamma}{2}} \notag\\
   &\leq \frac{2-\gamma}{2}\frac{1}{\cos \theta}N^{\frac{1}{2}} (\max_{\mathcal{C}_{\theta}} h)^{1-\frac{\gamma}{2}}\leq C_{\theta,\gamma}N^{\frac{1}{2}}\left(\frac{C''}{a}\right)^{\frac{2-\gamma}{2}},\notag
   \end{split}
\end{equation}
where $C_{\theta,\gamma}$ is a positive constant that may change from line to line, and depends only on $\theta,\gamma$. Thus
\begin{equation}
\begin{split}
p\left(\tilde{t}\left(\frac{a}{b}\right)^{\frac{2-\gamma}{2}} \right)&\leq p(0)+C_{\theta,\gamma}N^{\frac{1}{2}}\left(\frac{C''}{a}\right)^{\frac{2-\gamma}{2}}\tilde{t}\left(\frac{a}{b}\right)^{\frac{2-\gamma}{2}}\\
&\leq \left(1+\tilde{t}C_{\theta,\gamma}N^{\frac{1}{2}}\right)\left(\frac{C''}{b} \right)^{\frac{2-\gamma}{2}}. \notag
\end{split}
\end{equation}
It follows that
\begin{equation}\label{UY}
    h(\xi)\leq \left(1+\tilde{t}C_{\theta,\gamma}N^{\frac{1}{2}}\right)^{\frac{2}{2-\gamma}}\frac{C''}{b}, \quad {\rm with} \ |\xi'|=\tilde{t}\left(\frac{a}{b}\right)^{\frac{2-\gamma}{2}}.
\end{equation}
Since $h(\xi)\geq \varpi (\xi)$, combining \eqref{KI} with \eqref{UY} yields
\begin{equation}
     \tilde{t}\left(\frac{1}{a}\right)^{\frac{\gamma}{2}}  \left(\frac{1}{b} \right)^{\frac{2-\gamma}{2}}\leq \left[1+ C''(1+\tilde{t}C_{\theta,\gamma}N^{\frac{1}{2}})^{\frac{2}{2-\gamma}}\right]\frac{1}{b}. \notag
\end{equation}
Choose
\[
t_0:=\min\left\{\frac12\sin\theta,\,C^{-1}_{\theta,\gamma}N^{-1/2}\right\}.
\]
Since $b/a\geq1$, this choice lies in the interval allowed above and satisfies $t_0C_{\theta,\gamma}N^{1/2}\leq1$. Hence
\begin{equation}\label{TRE}
\frac{b}{a}\leq
\left[\frac{1+C''2^{\frac{2}{2-\gamma}}}{t_0}\right]^{\frac{2}{\gamma}}.
\end{equation}
On the other hand,
\begin{equation}
\varpi(\xi)\geq (\xi_{n+1}+\cos \theta)\left(\frac{1}{b}-\tau^{\ast}\right)\geq \cos \theta \left(\frac{1}{b}-\tau^{\ast}\right)=\cos \theta\frac{1-\lambda}{b} \notag.
\end{equation}
By a direct computation,
\[
\lambda':=\lambda'_{\eta}=\frac{\cos \theta \sin^2 \theta}{(\eta^2\cos^2 \theta+\sin^2 \theta)^{3/2}}>0.
\]
Since $\eta \leq1$, we have $\lambda\leq  \cos \theta$. It follows that
\begin{equation}\label{TR}
     \varpi(\xi)\geq  \frac{\cos\theta(1-\cos \theta)}{b}.
\end{equation}
Using \eqref{hui} and \eqref{TR}, we get
\begin{equation*}\label{pin-1}
    \frac{\max_{\mathcal{C}_{\theta}}h}{\min_{\mathcal{C}_{\theta}}h}\leq 
    \frac{bC''}{a \cos \theta(1-\cos \theta)}.
\end{equation*}
Combining this with \eqref{TRE} yields \eqref{maxmin}. This completes the proof.
\end{proof}

With the aid of Lemma \ref{TREQ}, we present a new approach to derive $C^0$ estimates for certain capillary curvature problems, relying on the corresponding gradient estimates.

We take the capillary $L_p$ curvature problem as an example. This problem seeks a capillary convex hypersurface whose capillary support function $h:\mathcal{C}_\theta\rightarrow \mathbb R$ satisfies the following fully nonlinear partial differential equation with a Robin boundary condition:
\begin{equation}\label{Mong-Eq25}
\left\{
\begin{array}{l@{\ }l}
\frac{\sigma_{n}}{\sigma_{n-k}}(\nabla^{2}h + hI) = f h^{p-1}, & \quad  {\rm in} \ \mathcal{C}_{\theta}, \\
\nabla_{\mu}h=\cot \theta \,h, & \quad  {\rm on} \ \partial\mathcal{C}_{\theta},
\end{array}\right.
\end{equation}
where $1 \leq k <n$ is an integer, and $\sigma_k(\nabla^{2}h + hI)$ is the $k$-th elementary symmetric function of the principal radii of curvature.

Recently, Y. Hu and Ivaki \cite{HI26} established the solvability of \eqref{Mong-Eq25} in the even case when $\theta\in (0,\frac{\pi}{2})$ and $1<p<k+1$, eliminating an extra contact-angle restriction previously required in \cite{MWW25}. The key to their proof lies in obtaining a gradient estimate which is valid for every $\theta\in (0,\pi/2)$, and then combining it with \eqref{Mong-Eq25} to derive the uniform $C^0$ and $C^2$ estimates simultaneously. This idea can be viewed as a capillary extension of the $L_p$ curvature problem in $\mathbb{R}^{n+1}$ with $1<p<k+1$ (see \cite{HI24}).

\begin{lem}\cite[Lem. 3.1]{HI26}\label{IUYQ}
Let $1<p<k+1$, $\theta\in (0, \frac{\pi}{2})$ and $0<\gamma< \frac{2(p-1)}{k}$. Suppose $h$ is an even, smooth, strictly convex solution of
\begin{equation*}
\left\{
\begin{array}{l@{\ }l}
\widetilde{F}(\nabla^{2}h + hI) = f h^{p-1}, & \quad  {\rm in} \ \mathcal{C}_{\theta}, \\
\nabla_{\mu}h=\cot \theta \,h, & \quad  {\rm on} \ \partial\mathcal{C}_{\theta},
\end{array}\right.
\end{equation*}
where $\widetilde{F}$ is a symmetric, $k$-homogeneous, smooth curvature function that is elliptic on the cone of positive definite matrices. Then there exists $C_0=C_0(\theta,k,p,\gamma,f)$ such that
\[
\frac{|\nabla h|^2}{h^{\gamma}}\leq C_{0} \left( \max_{\mathcal{C}_{\theta}}h\right)^{2-\gamma}.
\]
\end{lem}

Combining Lemma \ref{TREQ} with Lemma \ref{IUYQ}, we present a new argument to obtain the following $C^0$ estimate for Eq. \eqref{Mong-Eq25} .
\begin{lem}\label{NEWC0}
Let $1<p<k+1$ and $\theta\in (0,\frac{\pi}{2})$. Suppose $h$ is an even, positive, smooth and strictly convex solution to Eq. \eqref{Mong-Eq25}. Then there exists a positive constant $C$ depending on $n,\theta,k,p,f$ such that
\begin{equation}\label{NEWH}
\frac{1}{C}\leq h \leq C.
\end{equation}
\end{lem}

\begin{proof}
Let $\gamma=(p-1)/k$ and $u=h/\ell$. Then
\begin{equation}\label{Mong-Eq6}
\left\{
\begin{array}{l@{\ }l}
\frac{\sigma_{n}}{\sigma_{n-k}}(\ell\nabla^2 u+\nabla u\otimes \nabla \ell +\nabla \ell\otimes \nabla u+uI)= f (u\ell)^{p-1}, & \quad  {\rm in} \ \mathcal{C}_{\theta}, \\
\nabla_{\mu}u=0, & \quad  {\rm on} \ \partial\mathcal{C}_{\theta}.
\end{array}\right.
\end{equation}
Let $\xi_0\in \mathcal{C}_{\theta}$ be a maximum point of $u$. Using the boundary condition in \eqref{Mong-Eq6}, we find
\[
\nabla u(\xi_0)=0, \quad \nabla^2 u(\xi_0)\leq 0.
\]
Then at $\xi_0$,
\[
\frac{\sigma_{n}}{\sigma_{n-k}}(\ell\nabla^2 u+\nabla u\otimes \nabla \ell +\nabla \ell\otimes \nabla u+uI)\leq \frac{\sigma_n}{\sigma_{n-k}}(uI)=\binom{n}{k}^{-1}u^k.
\]
Using \eqref{Mong-Eq6}, at $\xi_0$, we have
\[
\binom{n}{k}^{-1}u^{k+1-p}\geq f \ell^{p-1}\geq (\min_{\mathcal{C}_{\theta}}f)(1-\cos \theta)^{p-1},
\]
where we used that $1-\cos \theta\leq \ell \leq \sin^2 \theta$. Thus
\begin{equation}\label{CNK}
(\max_{\mathcal{C}_{\theta}}h)^{k+1-p}\geq C\min_{\mathcal{C}_{\theta}}f,
\end{equation}
where $C$ is a positive constant that may change from line to line and depends only on $n,k,\theta,p$.

Similarly, we also obtain
\begin{equation}\label{CNKL}
(\min_{\mathcal{C}_{\theta}}h)^{k+1-p}\leq C\max_{\mathcal{C}_{\theta}}f.
\end{equation}

Lemmas \ref{TREQ} and \ref{IUYQ} imply
\begin{equation}\label{OIL}
     \frac{\max_{\mathcal{C}_{\theta}}h}{\min_{\mathcal{C}_{\theta}}h}  \leq C.
\end{equation}
Combining \eqref{OIL} with \eqref{CNKL} and \eqref{CNK}, we conclude that \eqref{NEWH} holds.
\end{proof}

The $C^0$ estimate yields $C^1$ and $C^2$ estimates, as well as higher-order regularity estimates for $h$ (see \cite{MWW25} for details). Applying the degree-theoretic argument, we then recover the existence result in \cite{HI26}.
\begin{rem}
In addition to the capillary $L_p$ curvature problem, Lemmas \ref{TREQ} and \ref{IUYQ} also allow us to recover the $C^0$ estimate for the even capillary $L_p$ Christoffel-Minkowski problem, provided that $1 < p < k+1$ and $\theta \in (0, \frac{\pi}{2})$ (see \cite{HI25}), by following an argument similar to that in Lemma \ref{NEWC0}.
\end{rem}

\section{Proof of Theorem \ref{maintheo2}}\label{sec5}
\subsection{Proof of Theorem \ref{maintheo2} (i)}
The key is to derive the $C^0$ estimate. By Lemma \ref{TREQ}, it suffices to obtain a gradient estimate for solutions to Eq. \eqref{Mong-Eq}.

\begin{lem}\label{gra-esi}
Let $\theta\in(0,\pi/2)$ and $p,q\in\R$.  Suppose that $h$ is a positive, smooth, strictly convex solution of Eq. \eqref{Mong-Eq}.  If
\begin{equation}\label{eq:gamma-condition}
 0<\gamma<2,
 \qquad
 \delta:=2(p-1)-\gamma(q-1)>0,
\end{equation}
then
\begin{equation}\label{Gra1}
 \frac{|\nabla h|^2}{h^\gamma}
 \le \widetilde C\left(\max_{\mathcal{C}_{\theta}}h\right)^{2-\gamma}.
\end{equation}
Here $\widetilde C$ depends on $n,\theta,p,q,\gamma,\min_{\mathcal{C}_\theta} f$ and $\|f\|_{C^1}$.
\end{lem}

\begin{proof}
Set $u=h/\ell$ and consider the auxiliary function
\[
 \Phi=\frac{\ell^{2-\gamma}|\nabla u|^2}{u^\gamma}.
\]
It suffices to prove
\begin{equation}\label{eq:u-gradient-target}
 \Phi\le M\left(\max_{\mathcal{C}_{\theta}}\ell\right)^{2-\gamma}
            \left(\max_{\mathcal{C}_{\theta}}u\right)^{2-\gamma}.
\end{equation}
Here $M>1$ will be chosen using only the data listed in the statement. If \eqref{eq:u-gradient-target} fails, let $\xi_0$ be a maximum point of $\Phi$.  Since $2-\gamma>0$, it follows that
\begin{equation}\label{eq:large-gradient}
\begin{split}
 \frac{|\nabla u|^2}{u^2}(\xi_0)
 ={}&\frac{\Phi(\xi_0)}{\ell(\xi_0)^{2-\gamma}u(\xi_0)^{2-\gamma}}\\
 >{}&M\left(\frac{\max_{\mathcal{C}_{\theta}}\ell}{\ell(\xi_0)}\right)^{2-\gamma}
    \left(\frac{\max_{\mathcal{C}_{\theta}}u}{u(\xi_0)}\right)^{2-\gamma}
 \ge M.
 \end{split}
\end{equation}

If $\xi_0\in\partial\mathcal{C}_\theta$, apply the boundary identities in \eqref{eq:fixed-cap-identities} and \eqref{eq:Neumann-identities}. Since $|\nabla u|(\xi_0)>0$ by \eqref{eq:large-gradient}, $\log\Phi$ is smooth near $\xi_0$.  At a boundary maximum, its outward normal derivative is nonnegative.  Consequently,
\[
 0\le(\log\Phi)_\mu
 =\frac{2u_\alpha u_{\alpha\mu}}{|\nabla u|^2}
   +(2-\gamma)\frac{\ell_\mu}{\ell}
   -\gamma\frac{u_\mu}{u}
 =-\gamma\cot\theta<0,
\]
which is a contradiction.  Thus $\xi_0$ is an interior point, i.e., $\xi_0 \in \mathcal{C}_\theta \backslash \partial \mathcal{C}_\theta$.

At $\xi_0$, choose an orthonormal frame diagonalizing $b_{ij}=h_{ij}+h\delta_{ij}$ and use the summation convention.  Set
\[
 \rho^2=h^2+|\nabla h|^2.
\]
At $\xi_0$, we have
\begin{equation}\label{eq:Phi-first}
 0=(\log\Phi)_i
 =\frac{2u_k u_{ki}}{|\nabla u|^2}
   +(2-\gamma)\frac{\ell_i}{\ell}
   -\gamma\frac{u_i}{u}, \quad i=1,\ldots,n.
\end{equation}
The Hessian $\nabla^2\log\Phi$ is negative semidefinite. Differentiation, together with $\ell_{ij}+\ell\delta_{ij}=\delta_{ij}$, leads to 
\begin{equation}
\begin{split}    
0
\geq (\log\Phi)_{ij}
={}&\frac{2(u_k u_{kij}+u_{ki}u_{kj})}{|\nabla u|^2}
 -\frac{(2-\gamma)(3-\gamma)}{\ell^2}\ell_i\ell_j\\
&+(2-\gamma)\gamma\frac{u_i\ell_j+u_j\ell_i}{u\ell}
 +(\gamma-\gamma^2)\frac{u_i u_j}{u^2}\\
&
 +(2-\gamma)\frac{1-\ell}{\ell}\delta_{ij}
 -\gamma\frac{u_{ij}}u.
\label{eq:Phi-second}
\end{split}
\end{equation}
Here \eqref{eq:Phi-first} was used to replace the square of the first derivative of $|\nabla u|^2$.

Let
\[
 \sigma_n=\det(b_{ij}),\qquad
 \sigma_n^{ij}=\frac{\partial\sigma_n}{\partial b_{ij}}
             =\sigma_n b^{ij}.
\]
Then $\sigma_n^{ij}$ is positive definite and
\begin{equation}\label{eq:cofactor-Euler}
 \sigma_n^{ij}b_{ij}=n\sigma_n.
\end{equation}
Since $h=u\ell$,
\begin{equation}\label{eq:b-u}
 b_{ij}=\ell u_{ij}+u_i\ell_j+u_j\ell_i+u\delta_{ij}.
\end{equation}
Then \eqref{eq:cofactor-Euler} and \eqref{eq:b-u} imply
\begin{equation}\label{eq:cofactor-uij}
 \sigma_n^{ij}u_{ij}
 =\frac n\ell\sigma_n-\frac2\ell\sigma_n^{ij}u_i\ell_j
  -\frac u\ell\sigma_n^{ij}\delta_{ij}.
\end{equation}
This explains, in particular, the contribution
\[
 -\frac{n\gamma\sigma_n}{h}
 =-n\gamma f h^{p-2}\rho^{n+1-q}
\]
in the contracted second derivative inequality.

For completeness, differentiating \eqref{Mong-Eq} once gives
\begin{equation}\label{eq:PDE-first}
 \sigma_n^{ij}\nabla_m b_{ij}
 =\sigma_n\left[
  (\log f)_m+(p-1)\frac{h_m}{h}
  +(n+1-q)\frac{hh_m+h_kh_{km}}{\rho^2}\right].
\end{equation}
The remaining substitutions use
\begin{align}
 u_{kij}&=u_{ijk}+u_k\delta_{ij}-u_j\delta_{ik},
 \label{eq:u-commute}\\
 u_m\ell_{jm}&=(1-\ell)u_j,\label{eq:ell-contract}\\
 u_mh_m&=\ell|\nabla u|^2+u u_m\ell_m,
 \label{eq:h-first-contract}\\
 u_mh_kh_{km}
 &=\ell u_mu_kb_{km}-h\ell|\nabla u|^2
   +u u_m\ell_kb_{km}-uh u_m\ell_m.
\label{eq:h-second-contract}
\end{align}
Differentiating \eqref{eq:b-u}, contracting with $u_m\sigma_n^{ij}$, using \eqref{eq:Phi-first} and \eqref{eq:ell-contract}, we get
\begin{equation}
\begin{split}
\sigma_n^{ij}u_m u_{ijm}
={}&\frac{\sigma_n}{\ell}u_m
 \left[(\log f)_m+(p-1)\frac{h_m}{h}\right] \\
&+\frac{(n+1-q)\sigma_n}{\ell\rho^2}
  u_m(hh_m+h_kh_{km})\\
&-\frac{u_m\ell_m}{\ell}\sigma_n^{ij}u_{ij}
 -\frac{2}{\ell}\sigma_n^{ij}(u_m u_{mi})\ell_j\\
&
 -\frac{2(1-\ell)}{\ell}\sigma_n^{ij}u_i u_j
 -\frac{|\nabla u|^2}{\ell}\sigma_n^{ij}\delta_{ij},
\label{eq:third-contraction-raw}
\end{split}
\end{equation}
where
\begin{equation}\label{eq:umumi}
 u_m u_{mi}=\frac{|\nabla u|^2}{2}
             \left(\gamma\frac{u_i}{u}
                   -(2-\gamma)\frac{\ell_i}{\ell}\right).
\end{equation}
The spherical commutator \eqref{eq:u-commute} further implies
\begin{equation}\label{eq:third-commuted}
 \sigma_n^{ij}u_m u_{mij}
 =\sigma_n^{ij}u_m u_{ijm}
   +|\nabla u|^2\sigma_n^{ij}\delta_{ij}-\sigma_n^{ij}u_i u_j.
\end{equation}

Substitute \eqref{eq:cofactor-uij}, \eqref{eq:h-first-contract}--\eqref{eq:third-commuted} into the contraction of \eqref{eq:Phi-second} with $\sigma_n^{ij}$.  The resulting inequality, with every term written explicitly, is
\begin{align}
0\ge{}&
 \frac{2h^{p-1}\rho^{n+1-q}}
      {\ell|\nabla u|^2}\,u_m f_m
 +\frac2{|\nabla u|^2}\sigma_n^{ij}u_{ki}u_{kj}\notag\\
&+\frac{2(p-1-n)fh^{p-1}\rho^{n+1-q}}
        {\ell^2|\nabla u|^2}\,u_m\ell_m
 +\frac{(2-\gamma)(\gamma-1)}{\ell^2}
       \sigma_n^{ij}\ell_i\ell_j\notag\\
&+\left(2(p-1)-n\gamma\right)f h^{p-2}\rho^{n+1-q}
 +\left(\frac{\gamma-\gamma^2}{u^2}
       -\frac{2(2-\ell)}{\ell|\nabla u|^2}\right)
       \sigma_n^{ij}u_i u_j\notag\\
&+\left(\frac{4u_m\ell_m}{\ell^2|\nabla u|^2}
       +\frac{2(2-\gamma)\gamma}{u\ell}\right)
       \sigma_n^{ij}u_i\ell_j+\left(\frac{2u u_m\ell_m}{\ell^2|\nabla u|^2}
       +\gamma\right)\sigma_n^{ij}\delta_{ij}\notag\\
&+\frac{2(n+1-q)fh^{p-1}\rho^{n-1-q}}
          {\ell|\nabla u|^2}
 \left(\ell u_mu_kb_{km}+u u_m\ell_kb_{km}\right).
\label{eq:contracted-full}
\end{align}

We next estimate the terms in \eqref{eq:contracted-full} that do not carry the factor $n+1-q$.  By \eqref{eq:large-gradient},
\begin{equation}\label{eq:small-ratio}
 \frac{u}{|\nabla u|}\le M^{-1/2}.
\end{equation}
By \eqref{eq:Phi-first} and Cauchy--Schwarz inequality in the $k$-index,
\begin{equation}\label{eq:Hessian-projection}
 \frac2{|\nabla u|^2}\sigma_n^{ij}u_{ki}u_{kj}
 \ge\frac12\sigma_n^{ij}
       \left(\gamma\frac{u_i}{u}-(2-\gamma)\frac{\ell_i}{\ell}\right)
       \left(\gamma\frac{u_j}{u}-(2-\gamma)\frac{\ell_j}{\ell}\right).
\end{equation}
The right-hand side of \eqref{eq:Hessian-projection} plus the three corresponding quadratic terms from \eqref{eq:contracted-full} is exactly
\begin{align}
&\frac12\sigma_n^{ij}
 \left(\gamma\frac{u_i}{u}-(2-\gamma)\frac{\ell_i}{\ell}\right)
 \left(\gamma\frac{u_j}{u}-(2-\gamma)\frac{\ell_j}{\ell}\right)\notag\\
&\quad+(2-\gamma)(\gamma-1)\sigma_n^{ij}
       \frac{\ell_i\ell_j}{\ell^2}
 +(\gamma-\gamma^2)\sigma_n^{ij}\frac{u_i u_j}{u^2}+2(2-\gamma)\gamma\sigma_n^{ij}
       \frac{u_i\ell_j}{u\ell}\notag\\
 ={}&\frac{(2-\gamma)\gamma}{2}\sigma_n^{ij}
   \left(\frac{u_i}{u}+\frac{\ell_i}{\ell}\right)
   \left(\frac{u_j}{u}+\frac{\ell_j}{\ell}\right)\ge0.
\label{eq:key-square}
\end{align}
Thus \eqref{eq:contracted-full} contains, in addition to \eqref{eq:key-square}, the positive term $\gamma\sigma_n^{ij}\delta_{ij}$.

To estimate the remaining terms, use \eqref{eq:fixed-cap-identities} and $|\nabla\ell|\le C$. Then $|u_m\ell_m|\le C|\nabla u|$ and
\[
 \sigma_n^{ij}\frac{\ell_i\ell_j}{\ell^2}
 \le C\sigma_n^{ij}\delta_{ij}.
\]
Writing $u_i/u=(u_i/u+\ell_i/\ell)-\ell_i/\ell$ and using the positive definiteness of $(\sigma_n^{ij})$, for every fixed $\tau>0$, we have
\begin{align*}
 \sigma_n^{ij}\frac{u_i u_j}{u^2}
 &\le2\sigma_n^{ij}
       \left(\frac{u_i}{u}+\frac{\ell_i}{\ell}\right)
       \left(\frac{u_j}{u}+\frac{\ell_j}{\ell}\right)
      +C\sigma_n^{ij}\delta_{ij},\\
 \left|\sigma_n^{ij}\frac{u_i\ell_j}{u\ell}\right|
 &\le\tau\sigma_n^{ij}
       \left(\frac{u_i}{u}+\frac{\ell_i}{\ell}\right)
       \left(\frac{u_j}{u}+\frac{\ell_j}{\ell}\right)
      +C_\tau\sigma_n^{ij}\delta_{ij}.
\end{align*}
The three residual coefficients satisfy
\begin{align*}
 \left|\frac{2(2-\ell)u^2}{\ell|\nabla u|^2}\right|
 \le \frac{C}{M}, \quad
 \left|\frac{4u(u_m\ell_m)}{\ell|\nabla u|^2}\right|
 +\left|\frac{2u(u_m\ell_m)}{\ell^2|\nabla u|^2}\right|
 \le \frac{C}{\sqrt M}.
\end{align*}
Hence the preceding two inequalities and Young's inequality imply that for $M>0$ sufficiently large, we have
\begin{equation}
\begin{split}
&-\frac{2(2-\ell)}{\ell|\nabla u|^2}\sigma_n^{ij}u_i u_j
 +\frac{4u_m\ell_m}{\ell^2|\nabla u|^2}\sigma_n^{ij}u_i\ell_j
 +\frac{2u u_m\ell_m}{\ell^2|\nabla u|^2}
   \sigma_n^{ij}\delta_{ij}\\
&\qquad\ge
 -\frac{(2-\gamma)\gamma}{4}\sigma_n^{ij}
  \left(\frac{u_i}{u}+\frac{\ell_i}{\ell}\right)
  \left(\frac{u_j}{u}+\frac{\ell_j}{\ell}\right)
 -\frac\gamma2\sigma_n^{ij}\delta_{ij}.
\label{eq:small-geometric}
\end{split}
\end{equation}
The two data terms in the first and third terms of \eqref{eq:contracted-full} are also small.  Indeed, using $h=u\ell$, $|u_mf_m|\le|\nabla u|\,|\nabla f|$, and $|u_m\ell_m|\le C|\nabla u|$, we obtain
\begin{equation}
\begin{split}
\left|
 \frac{2h^{p-1}\rho^{n+1-q}}
      {\ell|\nabla u|^2}\,u_m f_m
 \right|
&\le \frac{2u}{|\nabla u|}\frac{|\nabla f|}{f}
 f h^{p-2}\rho^{n+1-q}\le \frac{C}{\sqrt M}
 f h^{p-2}\rho^{n+1-q},\label{eq:f-error}
\end{split}
\end{equation}
and
\begin{equation}
\begin{split}
\left|
 \frac{2(p-1-n)fh^{p-1}\rho^{n+1-q}}
      {\ell^2|\nabla u|^2}\,u_m\ell_m
 \right|
&\le C\frac{u}{|\nabla u|}
 f h^{p-2}\rho^{n+1-q}\le \frac{C}{\sqrt M}
 f h^{p-2}\rho^{n+1-q}.
\label{eq:ell-error}
\end{split}
\end{equation}
Combining \eqref{eq:key-square}--\eqref{eq:ell-error} shows that the sum of all the terms in \eqref{eq:contracted-full} except its last, $q$-dependent term is bounded below by
\begin{equation}\label{eq:q-independent-estimate}
 \left(2(p-1)-n\gamma-\frac C{\sqrt M}\right)
 f h^{p-2}\rho^{n+1-q}.
\end{equation}

It remains to estimate the term containing $(n+1-q)$ from both sides. By \eqref{eq:b-u} and \eqref{eq:Phi-first},
\begin{align}
\ell u_mu_kb_{km}
={}&\frac{\gamma\ell^2}{2u}|\nabla u|^4
 +\left(1+\frac\gamma2\right)\ell|\nabla u|^2u_m\ell_m
 +u\ell|\nabla u|^2,
\label{eq:T1}\\
u u_m\ell_kb_{km}
={}&\left(\frac{\gamma\ell}{2}|\nabla u|^2+u^2\right)u_m\ell_m
 +\frac{\gamma u}{2}|\nabla u|^2|\nabla\ell|^2
 +u(u_m\ell_m)^2.
\label{eq:T2}
\end{align}
Since $h=u\ell$, substituting \eqref{eq:T1}--\eqref{eq:T2} shows that the last term in \eqref{eq:contracted-full} equals
\begin{equation}
\begin{split}
&2(n+1-q)\frac{u}{\rho^2}
 \left[
  \frac{\gamma\ell^2}{2u}|\nabla u|^2
  +(1+\gamma)\ell u_m\ell_m
  +\frac{u^2u_m\ell_m}{|\nabla u|^2}\right.\\
&\hspace{33mm}\left.
  +\ell u+\frac{\gamma u}{2}|\nabla\ell|^2
  +\frac{u}{|\nabla u|^2}(u_m\ell_m)^2
 \right]\cdot
 f h^{p-2}\rho^{n+1-q}.
\label{eq:q-term-expanded}
\end{split}
\end{equation}

The following two expansions compare its leading coefficient with $\gamma$:
\begin{align}
\frac{\rho^2}{\ell^2|\nabla u|^2}
&=1+\frac{2u u_m\ell_m}{\ell|\nabla u|^2}
 +\frac{u^2}{|\nabla u|^2}
  \left(1+\frac{|\nabla\ell|^2}{\ell^2}\right)
 =1+O\left(\frac{u}{|\nabla u|}
             +\frac{u^2}{|\nabla u|^2}\right),\label{eq:rho-ratio}
\end{align}
and
\begin{equation}
\begin{split}
&\frac{2u}{\ell^2|\nabla u|^2}
 \left[
  \frac{\gamma\ell^2}{2u}|\nabla u|^2
  +(1+\gamma)\ell u_m\ell_m
  +\frac{u^2u_m\ell_m}{|\nabla u|^2}\right.\\
&\hspace{31mm}\left.
  +\ell u+\frac{\gamma u}{2}|\nabla\ell|^2
  +\frac{u}{|\nabla u|^2}(u_m\ell_m)^2
 \right] \\
&\qquad=\gamma+O\left(
 \frac{u}{|\nabla u|}+\frac{u^2}{|\nabla u|^2}
 +\frac{u^3}{|\nabla u|^3}\right).
\label{eq:q-ratio-ell}
\end{split}
\end{equation}
Since $|u_m\ell_m|\le C|\nabla u|$, \eqref{eq:small-ratio} bounds both remainders in \eqref{eq:rho-ratio} and \eqref{eq:q-ratio-ell} by $CM^{-1/2}$. For large $M$, the ratio in \eqref{eq:rho-ratio} lies between $1/2$ and $3/2$. Dividing the two expansions therefore shows that the coefficient of $(n+1-q)fh^{p-2}\rho^{n+1-q}$ in \eqref{eq:q-term-expanded} equals $\gamma+O(M^{-1/2})$. The constants in these two remainders depend only on $\theta$ and $\gamma$, and are independent of $h$. For either sign of $n+1-q$, the last term in \eqref{eq:contracted-full} is consequently bounded below by
\[
 \left((n+1-q)\gamma
       -\frac{C|n+1-q|}{\sqrt M}\right)
 f h^{p-2}\rho^{n+1-q}.
\]
Combining this with \eqref{eq:q-independent-estimate} in \eqref{eq:contracted-full},
\[
 0\ge
 \left(2(p-1)-\gamma(q-1)-\frac C{\sqrt M}\right)
 f h^{p-2}\rho^{n+1-q}.
\]
Absorb $|n+1-q|$ into $C$ and choose $M$ above all preceding thresholds with $M>(2C/\delta)^2$. The coefficient is then at least $\delta/2>0$, this is a contradiction. Thus \eqref{eq:u-gradient-target} holds.

Using $\nabla h=\ell\nabla u+u\nabla\ell$, the positive upper and lower bounds for $\ell$, and $|\nabla\ell|\le C$, we obtain
\begin{align*}
 \frac{|\nabla h|^2}{h^\gamma}
 &\le C\left(
   \ell^{2-\gamma}\frac{|\nabla u|^2}{u^\gamma}
   +\ell^{-\gamma}u^{2-\gamma}|\nabla\ell|^2\right)\\
 &\le C\left(\frac{|\nabla u|^2}{u^\gamma}+u^{2-\gamma}\right)\\
 &\le C\left(\max_{\mathcal{C}_{\theta}}u\right)^{2-\gamma}
 \le C\left(\max_{\mathcal{C}_{\theta}}h\right)^{2-\gamma}.
\end{align*}
Here the inequality follows from \eqref{eq:u-gradient-target}, and the last one uses $\max_{\mathcal{C}_{\theta}}u\le(\min_{\mathcal{C}_{\theta}}\ell)^{-1}\max_{\mathcal{C}_{\theta}}h$. Then the desired inequality \eqref{Gra1} follows.
\end{proof}

\begin{rem}\label{rem:gradient-range}
There exists $0<\gamma<2$ satisfying \eqref{eq:gamma-condition} if and only if $p>\min\{1,q\}$. For $q>1$, this requires $p>1$ and $0<\gamma<\min\{2,2(p-1)/(q-1)\}$. For $q=1$, it requires $p>1$. For $q<1$, the supremum of $2(p-1)-\gamma(q-1)$ over $0<\gamma<2$ is $2(p-q)$, so the condition is equivalent to $p>q$. In particular, under $p\leq q$, Lemma \ref{gra-esi} applies precisely when $1<p\leq q$. If $p<q<1$, then
\[
2(p-1)-\gamma(q-1)<2(p-q)<0,
\]
so this argument does not extend to that range. The gradient estimate itself does not require evenness; that assumption enters through Lemma \ref{TREQ}. A fixed positive lower bound for $\delta$ is needed to absorb the error terms uniformly.
\end{rem}

\begin{lem}\label{C0p}
Let $n\geq2$, $1<p<q<\infty$ and $\theta\in (0,\frac{\pi}{2})$. Suppose $h$ is a positive, even, smooth and strictly convex solution of \eqref{Mong-Eq}. There is a constant $C$, depending only on $n,\theta,p,q,\min_{\mathcal{C}_\theta} f$ and $\|f\|_{C^1}$, such that
\begin{equation}\label{TRPA}
C^{-1}\leq h\leq C,\qquad |\nabla h|\leq C.
\end{equation}
\end{lem}
\begin{proof}
Choose $\gamma=(p-1)/(q-1)\in(0,1)$, so that $\delta=p-1>0$. Lemmas \ref{gra-esi} and \ref{TREQ} establish the ratio bound \eqref{maxmin}.  Put $u=h/\ell$ and
\begin{equation}\label{eq:extremum-weight}
\begin{aligned}
a(\xi)&=f(\xi)\ell(\xi)^{p-1}
       (\ell^2+|\nabla\ell|^2)^{(n+1-q)/2},\\
a_-&:=\min_{\mathcal{C}_{\theta}} a>0,\qquad a_+:=\max_{\mathcal{C}_{\theta}} a.
\end{aligned}
\end{equation}
By \eqref{eq:Neumann-identities}, we conclude that $\nabla u=0$ and $\nabla^2u\leq0$ at a maximum point of $u$ on $\mathcal{C}_{\theta}$, with the reversed Hessian inequality at a minimum point. Consequently, \eqref{eq:b-u} implies $b\leq uI$ at a maximum point of $u$ and $b\geq uI$ at a minimum point of $u$. At these extremal points, \eqref{Mong-Eq} becomes $\det b = au^{n+p-q}$, 
\begin{equation}\label{PQU}
(\max_{\mathcal{C}_{\theta}} u)^{q-p}\geq a_-,
\end{equation}
and
\begin{equation}\label{PQL}
(\min_{\mathcal{C}_{\theta}} u)^{q-p}\leq a_+.
\end{equation}
Note that \eqref{PQU} and \eqref{PQL} hold for all real $p,q$ without even assumptions. In the present case $(q-p)>0$; the fixed positive bounds for $\ell$ give a positive lower bound for $\max_{\mathcal{C}_{\theta}} h$ and an upper bound for $\min_{\mathcal{C}_{\theta}} h$. Together with \eqref{maxmin}, this proves the upper bound for $\max_{\mathcal{C}_{\theta}} h$ and lower bound for $\min_{\mathcal{C}_{\theta}} h$. Combining with \eqref{Gra1}, we have the upper bound of $|\nabla h|$.  All bounds are uniform. Hence we complete the proof of \eqref{TRPA}.
\end{proof}

Next, we derive a $C^2$ estimate under the $C^0$ and $C^1$ estimates.
\begin{lem}\label{C2p}
Let $n\geq2$, $\theta\in (0,\frac{\pi}{2})$, and let $p,q\in\mathbb R$ be fixed. Suppose $h$ is a positive, smooth and strictly convex solution of \eqref{Mong-Eq} satisfying
\[
0<m_0\leq h\leq M_0,\qquad |\nabla h|\leq M_1.
\]
Then
\[
C^{-1}I\leq\nabla^2h+hI\leq CI,
\]
where $C$ depends only on $n,\theta,p,q,m_0,M_0,M_1,\min_{\mathcal{C}_{\theta}} f$ and $\|f\|_{C^2}$. No evenness assumption or sign restriction on $(p-1)$ or $(n+1-q)$ is required.
\end{lem}

\begin{proof}
By the assumed bounds and \eqref{Mong-Eq}, there are constants $d_-,d_+>0$ such that
\begin{equation}\label{eq:conditional-det}
d_-\leq\det b=f h^{p-1}(h^2+|\nabla h|^2)^{(n+1-q)/2}\leq d_+.
\end{equation}
Thus, it suffices to obtain an upper bound for the trace of $b$; the lower bound for its determinant will then yield a lower bound for its smallest eigenvalue.

For simplicity, we set
\[
F:=\log \det b=(p-1)\log h+G, \quad F^{ij}=\frac{\partial F}{\partial b_{ij}}=b^{ij},
\]
where $\{b^{ij}\}$ is the inverse of $\{b_{ij}\}$ and
\[
 G:=\log f+\frac{n+1-q}{2}\log (h^2+|\nabla h|^2).
\]

Recall the auxiliary function
\begin{equation*}
Q=\log \sigma_1+Ah+B|\nabla u|^{2},
\end{equation*}
where $\sigma_1:=\sigma_{1}(b_{ij}(\xi))$ is the trace of the matrix $\{b_{ij}\}$ with $b_{ij}=h_{ij}+h\delta_{ij}$, $u=h/ \ell$, $A<0$ and $B>0$ are constants to be chosen later. Assume that $\max_{\mathcal{C}_{\theta}}Q(\xi)$ is attained at a point $\xi_{0}\in \mathcal{C}_{\theta}$.

{\bf Case 1:} $\xi_{0}\in \partial\mathcal{C}_{\theta}$. Choose an orthonormal frame $\{e_i\}^{n}_{i=1}$ around $\xi_{0}\in \partial\mathcal{C}_{\theta}$ such that $e_n=\mu$ at $\xi_{0}$ and $\{b_{ij}\}$ is diagonal at $\xi_{0}$. At $\xi_0$, the identities \eqref{eq:Neumann-identities} imply
\begin{align}
\nabla_{\mu}Q&=\frac{(\sigma_1)_{\mu}}{\sigma_1}+Ah_{\mu}-2B\cot\theta\, |\nabla u|^2\notag\\
&= \frac{\nabla_{\mu}b_{\mu\mu}+\cot\theta \sum_{\beta}(b_{\mu\mu}-b_{\beta\beta})}{\sigma_1}+A\cot \theta\, h -2B\cot\theta \,|\nabla u|^2,\label{bou-1}
\end{align}
where we used $h_{\mu\beta}=0$ and 
\begin{align}\label{eq-boundary-derivative}
\nabla_{\mu}b_{\beta\beta}=\cot \theta(b_{\mu\mu}-b_{\beta\beta}), \quad \text{for all $1\leq \beta \leq n-1$}.
\end{align}
On the other hand, using \eqref{eq-boundary-derivative} we have
\begin{align*}
F_{\mu}&=(\log \det b)_{\mu}=b^{ij}\nabla_{\mu}b_{ij}\\
&=b^{\mu\mu}\nabla_{\mu}b_{\mu\mu}+\sum_{\beta}b^{\beta\beta}\nabla_{\mu}b_{\beta\beta}\\
&=b^{\mu\mu}\nabla_{\mu}b_{\mu\mu}+\cot \theta \sum_{\beta}b^{\beta\beta}(b_{\mu\mu}-b_{\beta\beta}).
\end{align*}
It follows that
\[
\nabla_{\mu}b_{\mu\mu}=b_{\mu\mu}F_{\mu}-\cot \theta \, b^2_{\mu\mu}\sum_{\beta}b^{\beta\beta}+(n-1)\cot \theta\, b_{\mu\mu}.
\]
This implies that
\begin{align*}
(\sigma_1)_{\mu}&=b_{\mu\mu}F_{\mu}-\cot \theta \, b^2_{\mu\mu}\sum_{\beta}b^{\beta\beta}+2(n-1)\cot \theta\, b_{\mu\mu}-\cot \theta \sum_{\beta}b_{\beta\beta}.
\end{align*}
On $\partial\mathcal{C}_{\theta}$, $h_{\alpha\mu}=0$ and $h_\mu=\cot\theta\,h$ give
\[
(h^2+|\nabla h|^2)_\mu
=2h_\mu(h+h_{\mu\mu})=2\cot\theta\,h b_{\mu\mu}.
\]
Thus
\[
F_\mu=(p-1)\cot\theta+(\log f)_\mu
 +(n+1-q)\frac{\cot\theta\,h b_{\mu\mu}}{h^2+|\nabla h|^2}.
\]

Using the $C^0$ and $C^1$ estimates, \eqref{bou-1} and $b_{\mu\mu}/\sigma_1\leq 1$, we further obtain
\begin{equation*}
\nabla_{\mu}Q\leq C_0+\frac{b^2_{\mu\mu}}{\sigma_1}(C_1-C_2\sum_{\beta}b^{\beta\beta})+A\cot \theta \, h.
\end{equation*}
If $\sum_{\beta}b^{\beta\beta} \geq C_1/C_2$, then $\nabla_{\mu}Q(\xi_0)$ is strictly negative provided that $A\leq A_1<0$ is sufficiently negative. Otherwise, $\sum_{\beta}b^{\beta\beta} \leq C_1/C_2$. Combining this with $b_{\beta\beta}\geq (\sum_{\beta}b^{\beta\beta})^{-1}$, the $C^0$ and $C^1$ estimates, we have
\begin{align*}
b_{\mu\mu}=\frac{\det b}{\Pi_{\beta}b_{\beta\beta}}\leq C_3\left(\frac{C_1}{C_2} \right)^{n-1}.
\end{align*}
Thus $\nabla_{\mu}Q(\xi_0)$ is also strictly negative provided that $A\leq A_2<0$ is sufficiently negative. This contradicts the assumption  that $\xi_{0}\in \partial\mathcal{C}_{\theta} $. Hence the maximum point of $Q$ must be attained at some point $\xi_{0}\in \mathcal{C}_{\theta} \backslash \partial\mathcal{C}_{\theta}$ if $A<0$ is sufficiently negative.

{\bf Case 2:} $\xi_{0}\in \mathcal{C}_{\theta} \backslash \partial\mathcal{C}_{\theta}$.  Choose an orthonormal frame $\{e_i\}^{n}_{i=1}$ around $\xi_{0}$ such that  $\{b_{ij}\}$ is diagonal at $\xi_0$.  Set
\[
P=F^{ij}u_{ki}u_{kj}.
\]
Recall the decomposition \eqref{eq:b-u}. Let $T_{ij}:=\ell_{i}u_{j}+\ell_{j}u_{i}+u\delta_{ij}$. The assumed $C^0$ and $C^1$ estimates imply that
\begin{equation*}
|T_{ij}|\leq C.
\end{equation*}
Using \eqref{eq:b-u} and $(x-y)^2\geq \frac{x^2}{2}-y^2$, we find that at $\xi_0$,
 \begin{equation}
\begin{split}\label{PE}
P=\sum_{k,i}\frac{1}{b_{ii}}u^2_{ki}&=\sum_{k,i}\frac{1}{b_{ii}}\frac{(b_{ki}-T_{ki})^2}{\ell^2}\\
&\geq \sum_{k,i}\frac{1}{b_{ii}}\left(\frac{b^2_{ki}}{2\ell^2}-\frac{T^2_{ki}}{\ell^2} \right)\\
&\geq \frac{1}{2\ell^2}\sigma_1-C\sum_{i}b^{ii}\\
&\geq c_{0}\sigma_1-C\sum_i b^{ii}.
\end{split}
\end{equation}
Here we used \eqref{eq:fixed-cap-identities}. Differentiating \eqref{eq:b-u}, we find
 \begin{equation}
\begin{split}\label{bijk}
b_{ijk}&=\ell u_{ijk}+\ell_k u_{ij}+\ell_{ik}u_j +\ell_i u_{jk} +\ell_{jk}u_{i}+\ell_j u_{ik}+u_{k}\delta_{ij}\\
&=\ell (u_{kij}-u_k \delta_{ij}+u_j\delta_{ki})+\ell_k u_{ij}+\ell_{ik}u_j +\ell_i u_{jk} +\ell_{jk}u_{i}+\ell_j u_{ik}+u_{k}\delta_{ij}\\
&=\ell u_{kij}+\ell_k u_{ij}+u_{j}\delta_{ik}+\ell_{i}u_{jk}+(1-\ell) u_i \delta_{jk}+\ell_j u_{ik}+(1-\ell)u_k \delta_{ij}.
\end{split}
\end{equation}
Since $F_k=b^{ij}b_{ijk}$, \eqref{bijk} implies
 \begin{equation}\label{RA}
\ell b^{ij}u_{kij}=F_{k}-\ell_k b^{ij}u_{ij}-2b^{ij}\ell_{i}u_{jk}+(\ell-1)u_k \sum_{i}b^{ii}-(2-\ell)b^{kj}u_{j}.
\end{equation}
For the remaining contraction, divide \eqref{eq:cofactor-uij} by $\sigma_n$ and use $\sigma_n^{ij}=\sigma_n b^{ij}$. Write $\mathcal L=b^{ij}\nabla^2_{ij}$, this identity together with \eqref{PE} and \eqref{RA} yields
\begin{equation}
\begin{split}\label{Luu}
\mathcal{L}(|\nabla u|^2)&=b^{ij}(2u_{k}u_{kij}+2u_{ki}u_{kj})\\
&=\frac{2}{\ell}\left(  \langle \nabla F, \nabla u\rangle-\langle \nabla \ell, \nabla u  \rangle b^{ij}u_{ij}-2b^{ij}\ell_i u_{jk}u_k \right)\\
&\quad+\frac{2(\ell-1)}{\ell}|\nabla u|^2\sum_i b^{ii}-\frac{2(2-\ell)}{\ell}\sum_i b^{ii}u^2_i+2b^{ij}u_{ki}u_{kj}\\
&=\frac{2}{\ell}\langle \nabla F, \nabla u\rangle-\frac{2}{\ell^2}\langle \nabla \ell, \nabla u\rangle\left( n-2b^{ij}\ell_i u_j-u\sum_i b^{ii}\right)\\
&\quad-\frac{4}{\ell^2}b^{ij}\ell_iu_{k}(b_{jk}-\ell_j u_k-\ell_k u_j-u\delta_{jk})\\
&\quad+\frac{2(\ell-1)}{\ell}|\nabla u|^2\sum_i b^{ii}-\frac{2(2-\ell)}{\ell}\sum_i b^{ii}u^2_i+2b^{ij}u_{ki}u_{kj}\\
&\geq 2c_0 \sigma_1+\frac{2}{\ell}\langle \nabla F, \nabla u\rangle-C\left(1+\sum_ib^{ii} \right).
\end{split}
\end{equation}
Since $b=\nabla^2h+hI$ is a Codazzi tensor on the unit sphere, the contracted commutation identity is
\[
\Delta b_{ij}=\nabla^2_{ij}\sigma_1+nb_{ij}-\sigma_1\delta_{ij}.
\]
It also follows directly from $\Delta h_{ij}=(\Delta h)_{ij}+2nh_{ij}-2(\Delta h)\delta_{ij}$.

Thus,
\begin{equation}\label{SIE}
\mathcal{L}\sigma_1=b^{ij}\Delta b_{ij}-n^2+\sigma_1\sum_i b^{ii}.
\end{equation}
Now differentiating $F=\log \det b$ twice, we have
\begin{equation}\label{FE}
\Delta F=b^{ij}\Delta b_{ij}-b^{ip}b^{qj}(\nabla_{k}b_{pq}\nabla_{k}b_{ij}).
\end{equation}
Substituting \eqref{FE} into \eqref{SIE}, at $\xi_0$, we obtain
\begin{equation*}
\mathcal{L}\sigma_1=\Delta F+\sigma_1 \sum_i b^{ii}-n^2+\sum_{i,j,k}b^{ii}b^{kk}(\nabla_{j}b_{ik})^2.
\end{equation*}
For each $i$, the Cauchy-Schwarz inequality implies
\[
\sum_{j,k}b^{kk}(\nabla_jb_{ik})^2
\geq\sum_jb^{jj}(\nabla_i b_{jj})^2
\geq\frac{(\nabla_i\sigma_1)^2}{\sigma_1}.
\]
It follows that
\begin{equation}\label{SIE2}
\mathcal{L} \log\sigma_1=\frac{\mathcal{L} \sigma_1}{\sigma_1}-\frac{b^{ij}(\sigma_1)_i (\sigma_1)_j}{\sigma^2_1}\geq \frac{\Delta F}{\sigma_1}+\sum_i b^{ii}-\frac{n^2}{\sigma_1}.
\end{equation}
Moreover, by a direct computation,
\begin{equation}\label{LHE}
\mathcal{L} h=b^{ij}h_{ij}=b^{ij}(b_{ij}-h\delta_{ij})=n-h\sum_{i}b^{ii}.
\end{equation}
Recalling $F=(p-1)\log h +G$, we find
\[
\frac{\Delta \log h}{\sigma_1}=\frac{1}{h}-\frac{n}{\sigma_1}-\frac{|\nabla h|^2}{h^2\sigma_1}.
\]
So
\begin{equation}\label{LHE2}
\frac{\Delta F}{\sigma_1}=(p-1)\frac{\Delta \log h}{\sigma_1}+\frac{\Delta G}{\sigma_1}=(p-1)\left(\frac{1}{h}-\frac{n}{\sigma_1}-\frac{|\nabla h|^2}{h^2\sigma_1} \right)+\frac{\Delta G}{\sigma_1}.
\end{equation}
From \eqref{Luu}, \eqref{SIE2}, \eqref{LHE} and \eqref{LHE2}, one sees that at $\xi_0$,
\begin{equation}
\begin{split}\label{LQ2}
0&\geq \mathcal{L}Q=\mathcal{L}\log \sigma_1+A\mathcal{L}h+B\mathcal{L}|\nabla u|^2\\
&\geq (p-1)\left(\frac{1}{h}-\frac{n}{\sigma_1}-\frac{|\nabla h|^2}{h^2\sigma_1} \right)+\frac{\Delta G}{\sigma_1}+\sum_i b^{ii}-\frac{n^2}{\sigma_1}\\
&\quad+A(n-h\sum_{i}b^{ii})+B\left[  2c_0 \sigma_1+\frac{2}{\ell}\langle\nabla u, \nabla G \rangle+\frac{2(p-1)}{h\ell}\langle \nabla h, \nabla u\rangle-C\left(1+\sum_ib^{ii} \right)\right]\\
&\geq 2c_0 B \sigma_1+(-Ah-CB)\sum_{i}b^{ii}-C\left(1+|A|+B+\frac{1}{\sigma_1}\right)+\Phi,
\end{split}
\end{equation}
where
\[
\Phi:=\frac{\Delta G}{\sigma_1}+\frac{2B}{\ell}\langle \nabla u, \nabla G \rangle.
\]
At $\xi_0$,
\[
\frac{\nabla \sigma_1}{\sigma_1}=-A\nabla h-B\nabla(|\nabla u|^2).
\]
Using this fact and recalling $\rho^2=h^2+|\nabla h|^2$, we find
\begin{align*}
\Delta (\rho^2) &=2h\sigma_1-2nh^2+2|\nabla^2h|^2+2\langle \nabla h, \nabla \sigma_1\rangle,\\
\ \Delta G&=\Delta \log f+\frac{n+1-q}{2}\left(\frac{\Delta (\rho^2)}{\rho^2} -\frac{|\nabla(\rho^2)|^2}{\rho^4}\right),
\end{align*}
and
 \begin{equation*}
\begin{split}
\Phi&=\frac{\Delta (\log f)}{\sigma_1}+\frac{2B}{\ell}\langle \nabla u,\nabla \log f\rangle \\
&\quad +(n+1-q)\left(  \frac{h}{\rho^2}-\frac{nh^2}{\sigma_1 \rho^2}+\frac{|\nabla^2 h|^2}{\sigma_1 \rho^2}-A\frac{|\nabla h|^2}{\rho^2}\right)\\
&\quad -\frac{n+1-q}{2}\frac{|\nabla (\rho^2)|^2}{\sigma_1 \rho^4}+B\frac{n+1-q}{\rho^2}\chi,
\end{split}
\end{equation*}
where
\[
\chi=\ell^{-1}\langle \nabla u,\nabla (\rho^2)\rangle-\langle \nabla h,\nabla |\nabla u|^2 \rangle.
\]
The potentially unbounded Hessian terms in $\chi$ cancel exactly. Since $\nabla (\rho^2)=2b\nabla h$ and $b=\ell\nabla^2u+T$, symmetry implies
\[
\chi=\frac2\ell\langle\nabla u,b\nabla h\rangle
      -2\langle\nabla h,\nabla^2u\nabla u\rangle
     =\frac2\ell\langle\nabla h,T\nabla u\rangle.
\]
The assumed $C^0,C^1$ bounds therefore imply $|\chi|\leq C$.

On the other hand, using the $C^0$ and $C^1$ estimates again,
\[
|\nabla (\rho^2)|^2\leq C(1+|\nabla^2h|^2),
\]
and at $\xi_0$, one sees
\[
\frac{|\nabla^2 h|^2}{\sigma_1}=\frac{\sum_i(b_{ii}-h)^2}{\sigma_1}\leq \sigma_1+\frac{nh^2}{\sigma_1}.
\]

From the above, we further have
\begin{equation}\label{phi}
\Phi \geq -C_0 \sigma_1-C_1\left(1+|A|+B+\frac{1}{\sigma_1}\right).
\end{equation}
Here $C_0,C_1$ are independent of the choices of $A$ and $B$; the factor $(n+1-q)$ has been estimated in absolute value. Similarly, the terms containing $(p-1)$ in \eqref{LQ2} use only $|p-1|$. Substituting \eqref{phi} into \eqref{LQ2}, at $\xi_0$, we obtain
\begin{equation*}
\begin{split}
0\geq (2c_0 B-C_0) \sigma_1+(-Ah-CB)\sum_{i}b^{ii}-C_1\left(1+|A|+B+\frac{1}{\sigma_1}\right).
\end{split}
 \end{equation*}
Choose $B>0$ sufficiently large so that $2c_0B= C_0+1$, and then choose $A<0$ sufficiently negative such that
\[
-A\min_{\mathcal{C}_{\theta}}h\geq CB+1,
\]
where $A$ also satisfies $A\leq A_1<0$ and $A \leq A_2 \leq 0$. Then at $\xi_0$,
\[
0\geq  e^{Q-Ah-B|\nabla u|^2}-C_{1}\frac{1}{e^{Q-Ah-B|\nabla u|^2}}-C_2\notag.
\]
Thus $Q(\xi_0)\leq C$. The boundedness of $Ah+B|\nabla u|^2$ and maximality of $Q(\xi_0)$ imply $\sigma_1\leq C$ everywhere. Each eigenvalue of $b$ is at most $C$, and \eqref{eq:conditional-det} then gives $\lambda_{\min}(b)\geq d_-C^{-(n-1)}$, proving the claim.
\end{proof}

For higher regularity, use the gnomonic coordinates introduced before \eqref{eq:gnomonic-Hessian}, with $\xi(y)=v(y)+\cos\theta\,e$. By \eqref{eq:gnomonic-Hessian}, Lemma \ref{C2p} bounds $H$ in $C^2$ norm and bounds the eigenvalues of $D^2H$ above and away from zero. Writing $R=\tan\theta$ and $N=y/R$ for the Euclidean outward unit normal to $\partial B_R$, the Robin condition becomes $D_NH=H/R$. By \eqref{eq:gnomonic-determinant}, \eqref{Mong-Eq} becomes
\begin{align*}
\log\det D^2H={}&(n+2)\log a+\log f(\xi(y))+(p-1)\log(aH)
 \\
 {}&+\frac{n+1-q}{2}\log\bigl(|DH|^2+(H-y\cdot DH)^2\bigr).
\end{align*}
The right-hand side is uniformly Lipschitz along solutions. Choose fixed $0<\lambda_0<\Lambda_0$ so that the Hessians $D^2H$ lie in the interior of $\mathcal K_0=\{A:\lambda_0 I\leq A\leq\Lambda_0 I\}$, and define
\[
\widehat F(M)=\inf_{A\in\mathcal K_0}
\bigl\{\log\det A+\operatorname{tr}(A^{-1}(M-A))\bigr\}.
\]
Concavity of $\log\det$ shows that $\widehat F(M)=\log\det M$ for $M\in\mathcal K_0$. As an infimum of affine functions, $\widehat F$ is concave on all symmetric matrices. Moreover, for $P\geq0$,
\[
\Lambda_0^{-1}\operatorname{tr}P
\leq\widehat F(M+P)-\widehat F(M)
\leq\lambda_0^{-1}\operatorname{tr}P.
\]
 Thus \eqref{Mong-Eq} has a global uniformly elliptic concave extension, with constants controlled by Lemma \ref{C2p}. Set $v=-H$ and define $\widetilde F(M):=-\widehat F(-M)$. Since $\widehat F$ is concave and uniformly elliptic, $\widetilde F$ is convex and uniformly elliptic with the same ellipticity constants. Writing the equation for $H$ as $\widehat F(D^2H)=\Psi(y)$, we have
\[
\left\{\begin{aligned}
&\widetilde F(D^2v)=-\Psi(y)\quad &\text{in }B_R, \\
&D_N v =\frac{v}{R}\quad &\text{on }\partial B_R.
\end{aligned}\right.
\]
The boundary operator is smooth and uniformly oblique. Therefore, \cite[Thm.~1.3]{LZ18}, together with interior Evans--Krylov regularity, yields a global $C^{2,\alpha_0}$ bound for $H$, and hence for $h$, for some $0<\alpha_0<1$. Starting from the $C^{2,\alpha_0}$ bound, iterative differentiation and oblique Schauder estimates \cite{Lie13} give $C^{m,\alpha_0}$ bounds for every integer $m\geq2$, and the embedding $C^{k+3,\alpha_0}\hookrightarrow C^{k+2,\alpha}$ then yields
\begin{equation}\label{higher}
\|h\|_{C^{k+2,\alpha}(\mathcal{C}_{\theta})}\leq C_{k,\alpha} \quad \text{for every $k\geq0$ and every $0<\alpha<1$}.
\end{equation}

We now prove Theorem \ref{maintheo2} (i) by a degree-theoretic argument. Fix $1<p<q<\infty$ and choose
\[
1<p_0<\min\{p,n+1\},\qquad q_{\max}=\max\{q,n+1\},\qquad
\gamma=\frac{p_0-1}{q_{\max}-1}\in(0,1).
\]
In the first stage, set
\[
f_s=(1-s)\ell^{1-p_0}+sf,\qquad
q_s=(1-s)(n+1)+sq,\qquad 0\leq s\leq1,
\]
and keep $p_0$ fixed. Then
\begin{equation}\label{eq:uniform-path-margins}
2(p_0-1)-\gamma(q_s-1)\geq p_0-1>0,\qquad
q_s-p_0\geq\min\{q,n+1\}-p_0>0.
\end{equation}
In the second stage, keep $q,f$ fixed and let $p_t=(1-t)p_0+tp$, $0\leq t\leq1$. Here
\[
2(p_t-1)-\gamma(q-1)\geq p_0-1>0,\qquad q-p_t\geq q-p>0.
\]
The uniform positive lower bounds and regularity bounds
for the prescribed data, together with the two parameter
margins, allow us to apply Lemmas \ref{TREQ}, \ref{gra-esi}, \ref{C0p}, \ref{C2p} and \eqref{higher} uniformly along both paths.

Fix $\alpha\in(0,1)$. Choose $m,a_0>0$ strictly smaller
than the uniform lower bounds for $h$ and the smallest
eigenvalue of $b$, respectively, and choose $R_0$
strictly larger than the uniform bound for
$\|h\|_{C^{4,\alpha}}$. The bounded open set
\[
\mathcal O=\{h\in C^{4,\alpha}_{\rm even}(\mathcal{C}_{\theta}):
h>m,\ \nabla^2h+hI>a_0I,\ \|h\|_{C^{4,\alpha}}<R_0\}
\]
works for both paths. The first elliptic-oblique pair is
\begin{equation}\label{Mong-Eqs}
\begin{split}
F_s(h)&=\det(\nabla^2h+hI)
       -f_sh^{p_0-1}(h^2+|\nabla h|^2)^{(n+1-q_s)/2},\\
G(h)&=h_\mu-\cot\theta\,h.
\end{split}
\end{equation}
For the second path, define $\widetilde F_t$ by \eqref{Mong-Eqs} with $(p_0,q_s,f_s)$ replaced by $(p_t,q,f)$; keep the same boundary operator. Any solution of $(F_s(h), G(h))=(0,0)$ in $\overline{\mathcal O}$ is positive and strictly convex, and elliptic regularity for oblique boundary value problems ensures that it is smooth. The a priori estimates then imply that it lies in
$\mathcal O$. Hence, there are no solutions on $\partial\mathcal O$. The pairs are uniformly elliptic-oblique, and the uniform $C^5$ bound makes their solution sets compact in $C^{4,\alpha}$ space.

At the beginning of the first path, due to  $b(\ell)=I$,  $h_0=\ell$ solves \eqref{Mong-Eq}. Here $q_0=n+1$ and $f_0=\ell^{1-p_0}$, so \eqref{Mong-Eq} at $s=0$ becomes $\det b(h)= (h/\ell)^{p_0-1}$. Since $1<p_0<n+1$, it is the unique, positive, even, strictly convex solution by \cite[Thm.~5.4, $k=n$, $l=0$]{MWW25}; see also \cite[Sec.~6]{HI25}. Its linearized pair is
\begin{equation}\label{eq:linearization-base}
\phi\longmapsto\left(\Delta\phi+n\phi-\frac{p_0-1}{\ell}\phi,
\quad\phi_\mu-\cot\theta\,\phi\right).
\end{equation}
Write $\int u:= \int_{\mathcal{C}_{\theta}} u \,d\xi$, where $d\xi$ is the induced volume element. For a kernel element $\phi$, it follows from integration by parts, $\Delta \ell+n\ell=n$ and the Robin boundary conditions for $\phi$, $\ell$ that
\[
0=\int \ell \left(\Delta \phi+n\phi-\frac{p_0-1}{\ell}\phi\right)=(n+1-p_0)\int \phi.
\]
Since $n+1-p_0>0$, we have $\int \phi=0$.

The capillary Alexandrov-Fenchel inequality \cite[Thm.~3.2]{MWWX25}, with the first function equal to (possibly sign-changing) $\phi$ and all remaining functions $\ell$, yields 
\[
0=\left(\int \phi \right)^2 \geq \frac{1}{n}\int \ell \int \phi(\Delta \phi+n\phi). 
\]
Since $\ell>0$, the last integral is non-positive. Multiplying the kernel equation by $\phi$ and integrating yields
\[
0\geq\int\phi(\Delta\phi+n\phi) =(p_0-1)\int\frac{\phi^2}{\ell}.
\] 
As $p_0>1$ and $\ell>0$, this implies $\phi=0$.

For surjectivity, put $\phi=\ell v$ and divide the linearized pair by $\ell$. The resulting pair is $\cN+\cC$, where
\[
\begin{aligned}
(\cN+\cC)v
&=\left(\Delta v+2\langle\nabla\log\ell,\nabla v\rangle
       +\frac{n+1-p_0}{\ell}v,\ v_\mu\right),\\
\cN v&=(\Delta v-v,\ v_\mu),\\
\cC v&=\left(2\langle\nabla\log\ell,\nabla v\rangle
       +\left(1+\frac{n+1-p_0}{\ell}\right)v,\ 0\right).
\end{aligned}
\]
The Neumann operator $\cN$ is an isomorphism from $C^{4,\alpha}(\cC_\theta)$ onto $C^{2,\alpha}(\cC_\theta)\times C^{3,\alpha}(\partial\cC_\theta)$. The first component of $\cC v$ is bounded in $C^{3,\alpha}$ on bounded $C^{4,\alpha}$ sets, and its boundary component is zero. Thus $\cC$ is compact into the same space by $C^{3,\alpha}(\cC_\theta)\Subset C^{2,\alpha}(\cC_\theta)$. Since $\ell$ is even, $\cN,\cC$ commute with horizontal reflection. Thus, uniqueness makes the inverse $\cN^{-1}$ preserve evenness. Therefore, the Fredholm index is zero on both the full and even spaces, and the trivial kernel proves invertibility.

For real $r,\beta$, extend the power factors in the gnomonic coordinates above by
\[
\begin{gathered}
h^r(h^2+|\nabla h|^2)^\beta
\longmapsto \rho(h)^r\bigl(\rho(h)^2+|\nabla h|^2\bigr)^\beta,
\end{gathered}
\]
where $\rho\in C^\infty(\mathbb R;(0,\infty))$ satisfies $\rho(t)=t$, for all $t\geq m/2$. Both paths are unchanged on the neighborhood $U=\{h:\min h>m/2\}\supset\overline{\mathcal O}$. 

Set $E=C^{4,\alpha}_{\rm even}(\cC_\theta)$ and join the two paths at $F_1=\widetilde F_0$:
\[
\mathcal P_\lambda=
\begin{cases}
(F_\lambda,G),&0\leq\lambda\leq1,\\
(\widetilde F_{\lambda-1},G),&1\leq\lambda\leq2.
\end{cases}
\]
Apply \cite[Thm.~2 and Sec.~3]{LLN17} on a bounded neighborhood of $\overline{\mathcal O}$ in $C^{4,\alpha}$, invariant under reflection, with 
\[
h>m/2, \quad \nabla^2 h+hI>(a_0/2)I, \quad \|h\|_{C^{4,\alpha}}<2R_0.
\]
The uniform bounds and common principal coefficients allow us to fix one sufficiently large auxiliary constant $N$ for the whole path. In the gnomonic coordinates, let $S,T$ and $L_{h,N}$ be the isomorphisms constructed in \cite[Sec.~3]{LLN17}, and define
\[
\begin{aligned}
\mathcal D(r,s)&=(Sr,Ts),\qquad Q_h=L_{h,N}^{-1}\mathcal D,\\
Q_h&:C^{2,\alpha}(B_R)\times C^{3,\alpha}(\partial B_R)
\longrightarrow C^{4,\alpha}(B_R).
\end{aligned}
\]
The construction gives a smooth family of linear isomorphisms $Q_h$, with 
\[
K_\lambda(h):=Q_h\mathcal P_\lambda(h)-h
\]
forming a continuous compact homotopy.

For horizontal reflection $Jh(y)=h(-y)$ for $|y|\leq \tan\theta$, reflection invariance of the auxiliary operators and uniqueness imply $Q_{Jh}J=JQ_h$. Thus $Q_h$ restricts to an isomorphism between the even spaces when $h\in E$. With the even data and this fixed construction, set
\[
\Phi_\lambda(h)=Q_h\mathcal P_\lambda(h)=h+K_\lambda(h),
\quad \Phi_\lambda J=J\Phi_\lambda
\]
and
\[
\deg_E(\mathcal P_\lambda,\mathcal O,0)
:=\deg_{\mathrm{LS},E}(\Phi_\lambda|_E,\mathcal O,0),
\quad 0\notin\Phi_\lambda(\partial_E\mathcal O),
\]
where $\partial_E\mathcal O$ denotes the boundary of $\mathcal O$ relative to the even subspace $E$ with respect to the $C^{4,\alpha}$ norm. Since $\mathcal P_0(\ell)=0$, differentiation at $\ell$ yields
\[
\begin{aligned}
D\Phi_0(\ell)[\psi]
&=Q_\ell D\mathcal P_0(\ell)[\psi], \quad \psi\in E.
\end{aligned}
\]
 The kernel and Fredholm arguments above show that $D\mathcal P_0(\ell)$ is an isomorphism from $E$ onto its even range. Since $Q_\ell$ maps the even space isomorphically onto $E$, $\left.D\Phi_0(\ell)\right|_E$ is also invertible. Therefore, the unique initial zero $\ell$ has local degree $\pm1$. Excision and homotopy invariance on $E$, together with $F_1=\widetilde F_0$, imply
\[
\begin{aligned}
\deg_E((\widetilde F_1,G),\mathcal O,0)
&=\deg_E((\widetilde F_0,G),\mathcal O,0)\\
&=\deg_E((F_1,G),\mathcal O,0)
 =\deg_E((F_0,G),\mathcal O,0)=\pm1.
\end{aligned}
\]
Thus \eqref{Mong-Eq} admits a positive, even, smooth, strictly convex solution, proving (i).

\subsection{Proof of Theorem \ref{maintheo2} (ii)}
\begin{lem}\label{C01}
Let $\theta\in (0,\frac{\pi}{2})$ and $p>q$ be fixed real numbers. A positive, smooth, strictly convex solution of Eq. \eqref{Mong-Eq} satisfies \eqref{TRPA}, with $C$ depending on $n$, $\theta$, $p$, $q$, $\min_{\mathcal{C}_{\theta}} f$ and $\max_{\mathcal{C}_{\theta}} f$. No evenness is required.
\end{lem}
\begin{proof}
Use $u=h/\ell$, \eqref{eq:extremum-weight}, \eqref{PQU} and \eqref{PQL}. Since $p-q>0$, we have
\[
\max_{\mathcal{C}_{\theta}} u\leq a_-^{-1/(p-q)},\qquad
\min_{\mathcal{C}_{\theta}} u\geq a_+^{-1/(p-q)}.
\]
The fixed bounds for $\ell$ and \eqref{eq:geometric-gradient-bound} establish \eqref{TRPA}.
\end{proof}

Lemmas \ref{C01}, \ref{C2p} and \eqref{higher} provide the a priori estimates for fixed $p>q$, without evenness. For the continuity method, consider
\[
f_s=(1-s)f_0+sf,\qquad
f_0=\ell^{1-p}(\ell^2+|\nabla\ell|^2)^{(q-n-1)/2}.
\]
The initial solution is $h=\ell$, and closedness follows from the uniform estimates. For openness, put $w=\log h$ and $\mathcal A(w)=\nabla^2w+\nabla w\otimes\nabla w+I$. For any real $p,q$ and positive $g$,  \eqref{Mong-Eq} is equivalent to
\begin{equation}\label{eq:logarithmic-PDE}
\left\{\begin{aligned}
&\log\det\mathcal A(w)-\frac{n+1-q}{2}\log(1+|\nabla w|^2)
 +(q-p)w =\log g,\\
&w_\mu =\cot\theta,
\end{aligned}\right.
\end{equation}
with $\mathcal A(w)>0$. Along the present path, take $g=f_s$. Its linearization has positive-definite principal coefficients, zero-order coefficient $(q-p)$, and homogeneous Neumann boundary condition. The maximum principle forces the kernel to be trivial; invertibility then follows from the Fredholm index-zero property. The continuity method therefore reaches $s=1$. Uniqueness follows from the following lemma.

\begin{lem}\label{Uniq}
Let $p>q$ and $\theta\in (0,\frac{\pi}{2})$. Then a positive, smooth, strictly convex solution of Eq. \eqref{Mong-Eq} is unique.

\end{lem}

\begin{proof}
Let $h_1$ and $h_2$ be two positive solutions of \eqref{Mong-Eq}. Define $w_i=\log h_i$ and $\Phi=w_1-w_2$. The boundary conditions imply $\Phi_\mu=0$. Assume that $\Phi$ attains its maximum at $\xi_0\in\mathcal{C}_\theta$. By \eqref{eq:Neumann-identities}, at $\xi_0$, we have $\mathcal A(w_1)\leq\mathcal A(w_2)$. Applying \eqref{eq:logarithmic-PDE} to the two solutions, with $g=f$, we obtain
\[
(q-p)\Phi
=\log\det\mathcal A(w_2)-\log\det\mathcal A(w_1)\geq0.
\]
Since $p>q$, we obtain $\max_{\mathcal{C}_\theta}\Phi\leq0$. Interchanging $h_1$ and $h_2$, we also have $\min_{\mathcal{C}_\theta}\Phi\geq0$. It follows that $h_1\equiv h_2$. The proof is complete.
\end{proof}

\subsection{Proof of Theorem \ref{maintheo2} (iii)}

Fix $p=q>1$. For $0<\varepsilon<1$, let $h_\varepsilon$ be the solution of \eqref{Mong-Eq} from part (ii) with parameters $(p+\varepsilon,p)$ and $f$. It is even by uniqueness and reflection. Put $u_\varepsilon=h_\varepsilon/\ell$. The weight \eqref{eq:extremum-weight} is uniformly bounded above and away from zero at $(p+\varepsilon,p,f)$. Thus \eqref{PQU} and \eqref{PQL}, with exponent difference $-\varepsilon$, imply the bound $a_{-}\leq u_\varepsilon^{-\varepsilon}\leq a_{+}$, where $a_{-}$ and $a_{+}$ are computed at $(p+\varepsilon,p,f)$. For $m_\varepsilon:=\min_{\mathcal{C}_{\theta}} h_\varepsilon$ and $C_\varepsilon:=m_\varepsilon^\varepsilon$, the two-sided bounds on $\ell$ imply 
\[
\frac{(\min_{\mathcal{C}_{\theta}}\ell)^\varepsilon}{a_+} \leq C_\varepsilon \leq \frac{(\max_{\mathcal{C}_{\theta}}\ell)^\varepsilon}{a_-}. 
\]

Normalize $\bar h_\varepsilon=h_\varepsilon/m_\varepsilon$. By homogeneity,
\begin{equation}\label{Mong-Eq-var-1}
\left\{\begin{aligned}
\det(\nabla^2\bar h_\varepsilon+\bar h_\varepsilon I)
 &=C_\varepsilon f\bar h_\varepsilon^{p-1+\varepsilon}
       (\bar h_\varepsilon^2+|\nabla\bar h_\varepsilon|^2)^{(n+1-p)/2},\\
(\bar h_\varepsilon)_\mu&=\cot\theta\,\bar h_\varepsilon.
\end{aligned}\right.
\end{equation}
Apply Lemma \ref{gra-esi} with exponents $(p+\varepsilon,p)$ and the fixed choice $\gamma=1$. Since $2(p+\varepsilon-1)-(p-1)=p-1+2\varepsilon\geq p-1>0$, Lemma \ref{TREQ} and the normalization $\min_{\mathcal{C}_{\theta}}\bar h_\varepsilon=1$ yield
\begin{equation}\label{haver}
1\leq\bar h_\varepsilon\leq C,\qquad
|\nabla\bar h_\varepsilon|\leq C.
\end{equation}
By Lemma \ref{C2p}, \eqref{higher}, \eqref{haver} and the bounds for $C_\varepsilon$, a subsequence satisfies $C_{\varepsilon_j}\to C^*>0$ and $\bar h_{\varepsilon_j}\to h$ in $C^{2,\alpha'}$, $0<\alpha'<\alpha<1$, as $\varepsilon_j\downarrow0$. Since $\|\bar h_\varepsilon^\varepsilon-1\|_{C^0} \leq C^\varepsilon-1\to0$, passing to the limit in \eqref{Mong-Eq-var-1} yields \eqref{p=q}. The limit is positive and even; the uniform lower curvature bound preserves strict convexity, and regularity ensures smoothness.

We now prove uniqueness in both cases. Let $(h_i,C_i^*)$, $i=1,2$, solve \eqref{p=q}. Put $w_i=\log h_i$, $A_i=\mathcal A(w_i)>0$ and $\alpha_p=(n+1-p)/2$. They satisfy \eqref{eq:logarithmic-PDE} with $q=p$ and $g=C_i^*f$. At a maximum point of $(w_1-w_2)$ on $\mathcal{C}_{\theta}$, the  argument in Lemma \ref{Uniq} ensures that  $A_1\leq A_2$. Hence $C_1^*\leq C_2^*$. At a minimum point,  the reverse inequality holds, so $C_1^*=C_2^*$.

For this common constant, set $\varphi=w_1-w_2$ and define
\[
A_t=(1-t)A_2+tA_1,\qquad
g_t=(1-t)\nabla w_2+t\nabla w_1,\qquad
a^{ij}=\int_0^1(A_t^{-1})^{ij}\,dt.
\]
Every $A_t$ is positive definite. Then we find
\[
a^{ij}\varphi_{ij}+b^i\varphi_i=0,\qquad \varphi_\mu=0,
\]
where
\[
b^i=a^{ij}\bigl((w_1)_j+(w_2)_j\bigr)
       -2\alpha_p\int_0^1\frac{(g_t)_i}{1+|g_t|^2}\,dt.
\]
The sum of the gradients in this coefficient follows from symmetry of $a^{ij}$ and the identity
\[
a^{ij}\bigl((w_1)_i(w_1)_j-(w_2)_i(w_2)_j\bigr)
=a^{ij}\bigl((w_1)_j+(w_2)_j\bigr)\varphi_i.
\]
These coefficients are bounded and $a^{ij}$ is uniformly positive. By the strong maximum principle and the Hopf boundary point lemma, $\varphi$ is constant and hence $h_1=c h_2$ for some $c>0$. The normalization $\min_{\mathcal{C}_\theta} h=1$ then implies $h_1=h_2$, completing the proof of (iii).

\section*{Acknowledgment} We would like to thank Mohammad N. Ivaki for his helpful comments and discussions on this work.

{\bf Conflict of interest:} The authors declare that they have no conflict of interest.

{\bf Data availability:} Data sharing was not applicable to this article and no datasets were generated or analyzed in this study.


\begin{thebibliography}{20}

\bibitem[Ale38]{A38}
A. D. Aleksandrov,
\textit{On the theory of mixed volumes. III. Extensions of two theorems of Minkowski on convex polyhedra to arbitrary convex bodies},
Mat. Sb. \textbf{3}(1938): 27--46.

\bibitem[Ale39]{A39}
A. D. Aleksandrov,
\textit{On the surface area measure of convex bodies},
Mat. Sb. \textbf{6}(1939): 167--174.

\bibitem[CH25]{CH25}
C. Cabezas-Moreno, J. Hu, \textit{The $L_p$ dual Christoffel-Minkowski problem for $1< p< q\leq k+1$ with $1\leq k\leq n$}, Calc. Var. Partial Differential Equations {\bf 64} (2025), no.~7, Paper No. 229, 29 pp.

\bibitem[CY76]{CY76}
S. Y. Cheng, S. T. Yau,
\textit{On the regularity of the solution of the $n$-dimensional Minkowski problem},
Comm. Pure Appl. Math. \textbf{29}(1976): 495--516.

\bibitem[DGLL26]{DGLL26}
 S. Ding, J. Gao, G. Li, M. Liu, \textit{The Anisotropic Capillary $L_p$-Minkowski Problem}, arXiv:2603.01066 (2026).

\bibitem[Gao26a]{G26a}
Y. Gao, \textit{The $L_p$ dual Minkowski problem for capillary hypersurfaces}, arXiv:2510.12804v2 (2026).

\bibitem[Gao26b]{G26b}
Y. Gao, \textit{The capillary $L_p$ dual Minkowski problem for $p>q$ in higher dimensions}, arXiv:2609.07158 (2026).

\bibitem[GG02]{GG02}
B. Guan, P. Guan, \textit{Convex hypersurfaces of prescribed curvatures}, Ann. of Math. (2) {\bf 156} (2002), no.~2, 655--673.

\bibitem[Gua23]{Gu23}
P. Guan, \textit{A weighted gradient estimate for solutions of $L^p$ Christoffel-Minkowski problem}, Math. Eng. {\bf 5} (2023), no.~3, Paper No. 067, 14 pp.

\bibitem[GM03]{GM03}
P. Guan, X.-N. Ma, \textit{The Christoffel-Minkowski problem. I. Convexity of solutions of a Hessian equation}, Invent. Math. {\bf 151} (2003), no.~3, 553--577.

\bibitem[HLYZ10]{HLYZ10}
C. Haberl, E. Lutwak, D. Yang, G. Zhang, \textit{The even Orlicz Minkowski problem}, Adv. Math. {\bf 224} (2010), no.~6, 2485--2510.

\bibitem[HHI25]{HHI25}
J. Hu, Y. Hu, M.~N. Ivaki, \textit{Capillary $L_p$ Minkowski Flows}, arXiv:2509.06110 (2025).

\bibitem[HI26]{HI26}
Y. Hu, M.~N. Ivaki, \textit{Capillary $L_p$-curvature problem}, arXiv:2602.21832 (2026).

\bibitem[HI26a]{HI26a}
Y. Hu, M.~N. Ivaki, Personal communication (2026).

\bibitem[HI25]{HI25}
Y. Hu, M.~N. Ivaki, \textit{Capillary $L_p$-Christoffel-Minkowski problem}, arXiv:2512.15464 (2025).

\bibitem[HI25a]{HI25a}
Y. Hu, M.~N. Ivaki, \textit{Capillary curvature images}, arXiv:2505.12921 (2025).

\bibitem[HI24]{HI24}
Y. Hu, M.~N. Ivaki, \textit{Prescribed $L_p$ curvature problem}, Adv. Math. {\bf 442} (2024), Paper No. 109566, 15 pp.

\bibitem[HIS25]{HIS25}
Y. Hu,  M.~N. Ivaki, J.  Scheuer, \textit{Capillary Christoffel-Minkowski problem}, arXiv:2504.09320 (2025).

\bibitem[HWYZ24]{HWYZ24}
Y. Hu, Y. Wei, B. Yang, T. Zhou,
\textit{A complete family of Alexandrov-Fenchel inequalities for convex capillary hypersurfaces in the half-space},
Math. Ann. \textbf{390}(2024): 3039--3075.

\bibitem[HLYZ16]{HLYZ16}
Y. Huang, E. Lutwak, D. Yang, G. Zhang, \textit{Geometric measures in the dual Brunn-Minkowski theory and their associated Minkowski problems}, Acta Math. {\bf 216} (2016), no.~2, 325--388.

\bibitem[HZ18]{HZ18}
Y. Huang, Y. Zhao, \textit{On the $L_p$ dual Minkowski problem}, Adv. Math. {\bf 332} (2018), 57--84.

\bibitem[LZ18]{LZ18}
D. Li, K. Zhang, \textit{Regularity for fully nonlinear elliptic equations with oblique boundary conditions}, Arch. Ration. Mech. Anal. {\bf 228} (2018), no.~3, 923--967.

\bibitem[LL26]{LL26}
G. Li, C. Liu, \textit{Capillary Orlicz-Minkowski flow in the upper half-space}, arXiv:2601.14659 (2026).

\bibitem[LLN17]{LLN17}
Y.~Y. Li, J. Liu, L. Nguyen, \textit{A degree theory for second order nonlinear elliptic operators with nonlinear oblique boundary conditions}, J. Fixed Point Theory Appl. {\bf 19} (2017), no.~1, 853--876.

\bibitem[Lie13]{Lie13}
G.~M. Lieberman, {\it Oblique derivative problems for elliptic equations}, World Sci. Publ., Hackensack, NJ, 2013.

\bibitem[Lut93]{L93}
E. Lutwak, \textit{The Brunn-Minkowski-Firey theory. I. Mixed volumes and the Minkowski problem}, J. Differential Geom. {\bf 38} (1993), no.~1, 131--150.

\bibitem[LYZ18]{LYZ18}
E. Lutwak, D. Yang,  G. Zhang, \textit{$L_p$ dual curvature measures}, Adv. Math. {\bf 329} (2018), 85--132.

\bibitem[KLS25]{KLS25}
W. Klingenberg, B. Lambert, J. Scheuer,
\textit{A capillary problem for spacelike mean curvature flow in a cone of Minkowski space},
J. Evol. Equ. \textbf{25}(2025): 15, 24 pp.

\bibitem[MWW25]{MWW25}
X. Mei, G. Wang, L. Weng, \textit{Prescribed $L_{p}$ curvature problem for convex capillary hypersurface}, arXiv:2512.16686 (2025).

\bibitem[MWW25a]{MWW25a}
X. Mei, G. Wang, L. Weng,
\textit{The capillary Minkowski problem},
Adv. Math. \textbf{469}(2025): 110230, 29 pp.

\bibitem[MWW25b]{MWW25b}
 X. Mei, G. Wang, L. Weng, \textit{The capillary $L_p$ Minkowski problem}, arXiv:2505.07746 (2025), Trans. Amer. Math. Soc. (in press).

\bibitem[MWW25c]{MWW25c}
X. Mei, G. Wang, L. Weng,
\textit{The capillary Christoffel-Minkowski problem}, Calc. Var. Partial Differential Equations {\bf 65} (2026), no.~6, Paper No. 186, 22 pp.

\bibitem[MWW24]{MWW24}
X. Mei, G. Wang, L. Weng, \textit{Prescribed $L_{p}$ quotient curvature problem and related eigenvalue problem}, arxiv:2402.12314 (2024).

\bibitem[MWWX25]{MWWX25}
X. Mei, G. Wang, L. Weng, C. Xia,
\textit{Alexandrov-Fenchel inequalities for convex hypersurfaces in the half-space with capillary boundary II},
Math. Z. \textbf{310}(2025): 71, 17 pp.

\bibitem[Min97]{M897}
H. Minkowski,
\textit{Allgemeine Lehrs\"atze \"{u}ber die convexen Polyeder},
Nachr. Ges. Wiss. G\"ottingen (1897): 198--219.

\bibitem[Min03]{M903}
H. Minkowski,
\textit{Volumen und Oberfl\"ache},
Math. Ann. \textbf{57}(1903): 447--495.

\bibitem[Nir53]{N53}
L. Nirenberg,
\textit{The Weyl and Minkowski problems in differential geometry in the large},
Comm. Pure Appl. Math. \textbf{6}(1953): 337--394.

\bibitem[Pog78]{P78}
A. Pogorelov,
\textit{The Minkowski multidimensional problem},
translated by V. Oliker, Scripta Series in Mathematics, Winston, Washington, DC, 1978.

\bibitem[Sch14]{S14}
R. Schneider, {\it Convex bodies: the Brunn-Minkowski theory}, second expanded edition, Encyclopedia of Mathematics and its Applications, 151, Cambridge University Press, Cambridge, 2014.

\bibitem[WWX24]{WWX24}
G. Wang, L. Weng, C. Xia,
\textit{Alexandrov-Fenchel inequalities for convex hypersurfaces in the half-space with capillary boundary},
Math. Ann. \textbf{388}(2024): 2121--2154.

\bibitem[WZ25]{WZ25}
X. Wang, B. Zhu, \textit{The capillary Orlicz-Minkowski problem}, J. Funct. Anal. {\bf 291} (2026), no.~8, Paper No. 111574, 25 pp.

\end{thebibliography}
\end{document}